\documentclass{aupl}  

\usepackage{
amsfonts,
latexsym,
amssymb,
verbatim,
mathrsfs,
}

\numberwithin{equation}{section}

\newcommand{\labbel}{\label}

\newtheorem{theorem}{Theorem}[section]
\newtheorem{lemma}[theorem]{Lemma}

\newtheorem{proposition}[theorem]{Proposition} 
 
\newtheorem{corollary}[theorem]{Corollary}

\newtheorem*{claim*}{Claim}

\newtheorem*{theorem*}{Theorem}
\newtheorem*{proposition*}{Proposition}
\newtheorem*{corollary*}{Corollary}
\newtheorem*{lemma*}{Lemma}
\newtheorem*{scholion*}{Scholion}

\theoremstyle{definition}

\theoremstyle{remark}
\newtheorem{remark}[theorem]{Remark}
\newtheorem{remarks}[theorem]{Remarks}
\newtheorem*{remark*}{Remark}
\newtheorem*{remarks*}{Remarks}
 
\newtheorem{example}[theorem]{Example}

\newtheorem*{observation*}{Observation}

 \allowbreak

 \allowdisplaybreaks[4]

\begin{document}
 
\title
{The distributivity spectrum of Baker's  variety}

\author{Paolo Lipparini} 
\address{Dipartimento Fornace di Matematica\\Viale della  Ricerca
 Scientifica\\Universit\`a di Roma ``Tor Vergata'' 
\\I-00133 ROME ITALY}
\urladdr{http://www.mat.uniroma2.it/\textasciitilde lipparin}

\keywords{Baker's variety; 
J{\'o}nsson distributivity spectrum;
Congruence identity; 
Relation identity; 
Near-unanimity term; Edge term}

\subjclass[2010]{08B10; 06B75}
\thanks{Work performed under the auspices of G.N.S.A.G.A. Work 
 supported by PRIN 2012 ``Logica, Modelli e Insiemi''.
The author acknowledges the MIUR Department Project awarded to the
Department of Mathematics, University of Rome Tor Vergata, CUP
E83C18000100006.}

\begin{abstract}
For every $n$, we evaluate the smallest $k$ such that 
the congruence inclusion  
$\alpha (\beta \circ_n \gamma  )
\subseteq 
\alpha \beta \circ _{k} \alpha \gamma $ 
holds in a variety of reducts
of lattices  introduced by K. Baker.
We also  study  
varieties with a near-unanimity term and
discuss identities dealing with reflexive and admissible relations.
\end{abstract} 

\maketitle

\section{Introduction} \labbel{intro} 

Baker \cite{B} considered
the variety  generated by 
polynomial  reducts of lattices in which 
the only basic  operation
is $b$ defined by 
 $ b(a,c,d) =a (c + d)$.
 Here  juxtaposition  denotes meet and
$+$  denotes join. In a few cases, for clarity, the meet of $a$ and  $c$ 
shall be denoted by $ a \cdot c$. 
We shall denote the above variety by $\mathcal B$
and we shall call it  \emph{the Baker's variety}, but let us mention
that \cite{B} contains a more general study of varieties 
which arise as  reducts of lattices; see, in particular, \cite[Theorem 2]{B}. 
Notice that, in every algebra in 
$\mathcal B$, the term 
$x \cdot y = b(x,y,y)$
provides a semilattice operation; in particular, we can consider 
any algebra in $\mathcal B$  
as an ordered set in a natural way.
A related variety
is obtained by taking polynomial reducts of lattices in which
 the only basic operation
is $u$ defined by
$u(a_1, a_2, a_3, a_4) = \prod _{j \neq j } (a_i + a_j) $,
where the indices on the product vary on the set
$\{ 1,2,3,4\}$.  
We shall denote this variety by $\mathcal N_4$.
Notice that $u$ is a near-unanimity term in $\mathcal N_4$ and  
that the position $b(a,c,d) = u(a,a,c,d)$
provides an interpretation of  $\mathcal B$ in
$\mathcal N_4$.

Baker showed  that $\mathcal B$ is $4$-distributive 
but not $3$-distributive.
Recall that a variety $\mathcal V$ is
\emph{$m$-distributive}, or $\Delta_m$,  if $\mathcal V$ satisfies the congruence 
identity
$ \alpha ( \beta \circ \gamma  ) \subseteq \alpha \beta 
\circ _m\alpha \gamma $. 
In the above formula, $\alpha, \beta, \dots$
are intended to be variables for congruences of
some algebra in $\mathcal V$,
juxtaposition denotes intersection and 
  we have used the 
 shorthand
$ \beta \circ _m \gamma $  for
$ \beta \circ \gamma \circ \beta \dots$
with $m$ factors, that is, 
with $m-1$ occurrences of $\circ$.
If, say, $m$ is odd, we sometimes  write
$ \beta \circ \gamma \circ  {\stackrel{m}{\dots}} \circ \beta $
in place of  $ \beta \circ _m \gamma $
in order to make clear that $\beta$ is the last factor.
Conventionally,
$ \beta \circ _0 \gamma  = 0$, the minimal congruence
of the algebra under consideration; 
otherwise the reader might always suppose that
$m \geq 1$. 
We refer to Baker \cite{B}, J{\'o}nsson  
\cite{CV} or  Lipparini \cite{jds} 
for other  unexplained notions and notations.

The  original definition
of $m$-distributivity  
involves the existence of a certain number of terms
introduced by J{\'o}nsson \cite{J}; J{\'o}nsson terms are exactly
the terms arising from the Maltsev condition associated 
to $ \alpha ( \beta \circ \gamma  ) \subseteq \alpha \beta 
\circ _m \alpha \gamma $.
Here it will be more convenient to 
express results by means of
 congruence identities rather than 
terms. See \cite{jds} for a more detailed 
discussion and further references.
J{\'o}nsson proved that a variety is distributive
if and only if it is $m$-distributive, for some $m$. 
It follows from J{\'o}nsson's proof that, for every $n$ and  $m$, 
there is some $k$ such that  every  $m$-distributive variety 
 satisfies the congruence  identity
 $\alpha (\beta \circ_n \gamma  )
\subseteq 
\alpha \beta \circ _{k} \alpha \gamma $.
We initiated the study of the related
``J{\'o}nsson distributivity spectra'' in 
\cite{jds}. Here we shall evaluate exactly the distributivity spectra 
of $\mathcal N_4$ and of Baker's variety $\mathcal B$.   
We shall also show that we get exactly the same spectra if we consider
the corresponding reducts of \emph{distributive} lattices, call such  reducts 
 $\mathcal N_4^d$ and $\mathcal B ^d$.   

Relying heavily on Kazda,  Kozik,   McKenzie and  Moore \cite{adjt},
 we  observed in \cite{jds} 
that congruence distributive varieties satisfy also identities
of the form 
$\alpha (R \circ_n T)
\subseteq \alpha R \circ _{k'} \alpha T$,
where $R$ and $T$ denote reflexive and admissible relations. 
In Sections \ref{relid} and \ref{two}
  we shall find the best bounds for identities of this kind
in $\mathcal B$ and  $\mathcal N_4$;
 moreover, we shall  show that in the case of 
$\mathcal B$ and  $\mathcal N_4$ it is possible to take
 $\alpha$, too, as an 
admissible relation.
As far as relation identities are concerned,
 $\mathcal B$ and
$\mathcal N_4$ exhibit a subtly different behavior.
This partially confirms 
the suggestion 
implicit in \cite[p.\ 370]{CV} and explicitly
advanced  in \cite{uar} 
 that  
the study of relation identities 
might provide a finer classification of varieties
(in particular, congruence distributive varieties),
 in comparison with the study of congruence identities alone.

The relation identities found in 
Sections \ref{relid} and \ref{two} solve also some earlier problems.
In \cite{contol} we have showed that, under a fairly general assumption,
 a congruence identity 
is equivalent to the same identity when considered 
for representable tolerances, instead.
In Remark \ref{contol!} we show that the assumption
of representability of tolerances is necessary in the above equivalence. 

It is known
 \cite{chl,ricm} that the identities
$\alpha ( \Theta  \circ \Theta) \subseteq \alpha \Theta \circ_{k'} \alpha \Theta $
and  $\alpha (R \circ R) \subseteq \alpha R \circ_k \alpha R $,
for some $k$, $k'$, both characterize congruence modularity,
where $\Theta$ denotes a tolerance.
Remarks \ref{tolmod} and \ref{tolmodrel} show that, for a variety $\mathcal V$,
the best values of $k$ or  $k'$ in the above identities are not determined 
by the Day modularity level of $\mathcal V$. It is an open problem to find 
the example of a variety for which the best values for 
$k$ and  $k'$ above are distinct.

Section \ref{nearu} contains a few remarks about 
relation identities 
satisfied by varieties with a near-unanimity term 
and by varieties with an edge term.
Here we are dealing with the general case, not
with specific examples such as  $\mathcal {N}_4$. 
Further remarks  are contained in Section \ref{fur}.
Among other, and
following the lines of \cite{B}, 
we consider identities satisfied by 
arbitrary polynomial reducts of lattices. 
We also consider polynomial reducts 
of Boolean algebras.

\section{The distributivity spectra of 
$\mathcal B$ and $\mathcal N_4$} \labbel{thed} 

Recall that
$\mathcal B$ is the variety  generated by 
polynomial  reducts of lattices in which 
the only basic  operation
is  
 $ b(a,c,d) =a (c + d)$ and 
that $\mathcal N_4$ is defined similarly with respect to  
the operation
$u(a_1, a_2, a_3, a_4) = \prod _{j \neq j } (a_i + a_j) $.
The varieties obtained 
by considering only reducts of distributive lattices are denoted,
correspondingly, by
 $\mathcal B ^d$ and  $\mathcal N_4^d$.   

\begin{theorem} \labbel{bds}
Suppose that  $n \geq 2$
and 
 $\mathcal V $ is either 
$ \mathcal B $, $ \mathcal B ^d $, $  \mathcal N_4$
or  $\mathcal N_4^d$. 
Then 
 $\mathcal V$ satisfies
the following
congruence 
identities:
\begin{align}
 \labbel{1}
\alpha (\beta \circ_n \gamma  )
& \subseteq 
\alpha \beta \circ _{2n} \alpha \gamma,  
&& \text{for $n$ even, and} 
\\
\labbel{4}
\alpha (\beta \circ_n \gamma  )
& \subseteq 
\alpha \beta \circ _{2n-1} \alpha \gamma , 
&& \text{for $n$ odd;} 
\end{align}     
 and the subscripts on the right-hand sides are  best possible;
actually, $\mathcal V$ does not even satisfy
\begin{align} \labbel{3}
\alpha (\beta \circ (\alpha \gamma \circ
 \alpha \beta \circ  {\stackrel{n-2}{\dots}} \circ \alpha \beta ) \circ  \gamma   ) 
&\subseteq
\alpha \beta \circ _{2n-1} \alpha \gamma, 
&& \text{for $n$ even, and} 
\\   
 \labbel{6}
\alpha (\beta \circ (\alpha \gamma \circ
 \alpha \beta \circ  {\stackrel{n-2}{\dots}} \circ \alpha \gamma) \circ  \beta  ) 
&\subseteq
\alpha \beta \circ _{2n-2} \alpha \gamma,
&& \text{for $n$ odd.}  
\end{align}   
 \end{theorem}

\begin{proof}
The positive result that equations \eqref{1} and \eqref{4}  hold 
in $\mathcal B$ is 
an observation in \cite[Section 3]{jds},
however inserted there in a quite abstract and general context.
In the special case of $\mathcal B$ the proof is direct and
is an almost immediate
generalization of Baker's argument.
Indeed, if $n$ is even and 
$(a,d) \in \alpha (\beta \circ_n \gamma  )$, then
$a \mathrel \alpha  d$ and 
there are elements $c_i$ such that  
 $a= c_0 \mathrel { \beta  } c_1 \mathrel { \gamma } c_2 
\mathrel { \beta  } \dots
c _{n-1} \mathrel { \gamma } c_n =d  $.
Then the elements 
\begin{equation} \labbel{B}     
\begin{gathered}  
a= b(a,a,d)=b(a,c_0, d), \quad  
b(a,c_1, d), \quad  b(a,c_2, d), \ \ \dots \ \  
b(a,c_{n-1}, d),
\\
  b(a,c_{n}, d)= b(a,d, d)= a \cdot d = b(d,a,a)= b(d,c_0,a),
\\
b(d,c_1, a), \quad  b(d,c_2, a), \ \ \dots \ \ 
  b(d,c_{n-1}, a), \quad b(d,c_n, a)=b(d,d, a)=d
 \end{gathered}
  \end{equation}    
witness 
$(a,d) \in \alpha \beta \circ_{2n} \alpha \gamma  $.
Notice that, say, 
$b(a,c_j, d) \mathrel \alpha  b(a,c_j, a) = a = b(a,c_{j+1}, a)
\mathrel \alpha  b(a,c_{j+1}, d)$, for every $j$. 
The same chain of elements works in the case $n$ odd, 
but in this case 
$c _{n-1} \mathrel { \beta  } c_n =d  $,
hence 
$b(a,c_{n-1}, d) \mathrel { \alpha \beta } b(a,d, d)= a \cdot d = b(d,a,a)
\mathrel { \alpha \beta } b(d,c_1, a)$,
in particular, $b(a,c_{n-1}, d) \mathrel { \alpha \beta } b(d,c_1, a)$,
thus one passage might be skipped 
 and we get 
$(a,d) \in \alpha \beta \circ_{2n-1} \alpha \gamma  $.
Since $\mathcal B$
is interpretable in  $ \mathcal N_4$,
then \eqref{1} and \eqref{4}
hold in  $ \mathcal N_4$, too. 
The result for 
$ \mathcal N_4$ can be obtained also directly from 
the case $m=2$ of equations
\eqref{15} and \eqref{16} in
Proposition \ref{sch} below. 
Clearly, if some congruence identity 
holds in $\mathcal {B}$, respectively,  $ \mathcal N_4$,
then it holds in $\mathcal {B}^d$, respectively,   $ \mathcal N_4^d$.

Now we show that equations \eqref{3} and \eqref{6} fail,
hence the bounds in   \eqref{1} and \eqref{4} are  optimal.
We shall present the argument for 
$ \mathcal N_4$ and $ \mathcal N_4^d$. This is enough, since,
say,
$\mathcal B$
is interpretable in  $ \mathcal N_4$.
In any case, the same argument works
for $\mathcal B$ and $ \mathcal B^d$, too, with no essential modification. 

For $h \geq 1$, let $\mathbf C_{h+1}$ denote the $h+1$-elements chain 
with underlying set 
$C_{h+1} = \{0,1, \dots, h \} $
and with the standard ordering, inducing
the standard lattice operations of $\min$ and $\max$. 
Let $\mathbf L $
be the lattice
$  \mathbf C_{n+1} \times \mathbf C_{n+1} 
\times \mathbf C_2 $.
Since in what follows
the ``last''  $\mathbf C_2$ will play a different role with respect to the
first two $\mathbf C_{n+1}$'s, we shall usually denote the largest elements of 
$\mathbf C_2$ by $\uparrow$ and the smallest element of 
$\mathbf C_2$ by $\downarrow$. 
Consider the following elements of $L$:
\begin{align*}
 a=c_0& =(n,0,\uparrow),   & d=c_{n} = (0,n,\uparrow),
\qquad  \text{ and } 
\\
 c_i&= (n-i,i, \downarrow),   &     \text{for $i=1, \dots,n-1$.} \qquad
\end{align*} 

Recall that we let $u(x_1, x_2, x_3, x_4) = \prod _{j \neq k } (x_j + x_k) $
and 
 $b(x_1, x_3, x_4) = \allowbreak  u(x_1, x_1, x_3, x_4)= 
x_1(x_3+x_4)$.
 Let 
\begin{equation*}\labbel{BB}    
B= \{ a \in L \mid a \leq c_i, \text{ for some } i \leq n  \} .
  \end{equation*}
We show that $B$ is closed under $u$, hence  
$\mathbf B = (B, u)$ is an algebra in $\mathcal {N}_4$, actually, 
in $\mathcal {N}_4^d$, since $\mathbf L$ is a distributive lattice.
Indeed, suppose that  $a_{1} \leq c_{i_1}$, \dots ,
$a_{4} \leq c_{i_4}$. Since $u$ is invariant under any permutation of its
arguments, it is no loss of generality to assume that
$i_1 \leq i_2 \leq i_3 \leq i_4 $.   
If $u(a_1,a_2,a_3,a_4)$ has a third $\uparrow $
component, then at least three among 
$a_1,a_2,a_3,a_4$ have a   third $\uparrow $
component, hence at least two among 
$a_1,a_2,a_3,a_4$ are either $ \leq c_0$ or  $ \leq c_n$. 
Say, $a_1,a_2 \leq  c_0$, hence, 
since $u$ is a monotone operation,
$u(a_1,a_2,a_3,a_4) \leq 
u(c_0,c_0,c_{i_3},c_{i_4}) =
c_0(c_{i_3}+c_{i_4}) \leq c_0$.
Otherwise
$u(a_1,a_2,a_3,a_4)$ 
has a   third $\downarrow $
component and
$u(a_1,a_2,a_3,a_4) \leq u(c_{i_1},c_{i_2},c_{i_3},c_{i_4})
=(n-i_3, i_2, \downarrow) =c_{i_2}c_{i_3} \leq c_{i_2}$. In any case,
$u(a_1,a_2,a_3,a_4) \in B$. 

Hence 
$\mathbf B = (B, u) $ is an algebra in $  \mathcal {N}_4$;
in particular,  $(B, b) $ is an algebra in $  \mathcal {B}$.
Let $B ^{ \uparrow } $ denote the set 
of all the elements of $B$
with 
a last $ \uparrow $.
By the definition of $B$,
the elements of 
$B^{ \uparrow } $ 
are exactly the following:
\begin{equation}\labbel{bu}\tag{$B ^{ \uparrow } $} 
e _{i} =  a(d+ c _{i}) = (n-i,0, \uparrow),  \quad  
f _{i} =  d(a+ c _{i}) = (0,i, \uparrow), \quad i=0, \dots, n .
 \end{equation}    
 
Now we can show that \eqref{3} and \eqref{6}  fail, in general.
Let $\alpha$ be the 
kernel of the third projection, thus
$\alpha$ is a congruence 
on $\mathbf B$.

Let $\beta$ be the congruence 
on $\mathbf B$ 
defined in such a way that two elements $(i_1,i_2,i_3) $
and
$(j_1,j_2,j_3) $ of $B$ 
are $\beta$-related
if and only if,
for every $\ell =1,2$, 
 their components $i_ \ell $ and $j_ \ell $  
differ at most by $1$,
and:

(a$_{ \beta }$) if $i_1 \neq j_1$, then
$ \sup \{ i_1 , j_1 \} $ has the same parity of $n$, and 

(b$_{ \beta }$) if $i_2 \neq j_2$, then
$ \sup \{ i_2, j_2 \} $ is odd.

It can be checked directly that $\beta$ is a congruence; otherwise, argue as
follows.
Let $\beta_1'$ be 
the congruence 
on $\mathbf C_{n+1}$ whose blocks are
$ \{ n, n-1\}, \{ n-2, n-3 \}, \dots $\ 
If $\beta''_1$ is the counterimage in 
 $\mathbf L$ of $\beta_1'$  through
the first projection, then 
$\beta''_1$ is a congruence on $\mathbf L$,
hence a congruence on $(L,u)$. 
Thus 
 the restriction $\beta_1$ of
$\beta''_1$ to $B$ is a congruence on 
$(B,u)$. 
Similarly, define $\beta_2$
using the counterimage through the second projection
of the congruence  
on $\mathbf C_{n+1}$ whose blocks are
$ \{ 0,1\},  \{2, 3\},  \dots $\    
 Then $\beta= \beta_1 \beta_2 $, hence $\beta$ 
is a congruence, being the meet of  two congruences.  

The congruence $\gamma$  
on $\mathbf B$ 
is defined in a similar way:  two elements $(i_1,i_2,i_3) $
and
$(j_1,j_2,j_3) $ of $B$ 
are $\gamma$-related
if and only if,
for every $\ell =1,2$, 
 their components $i_ \ell $ and $j_ \ell $  
differ at most by $1$,
and:

(a$_{ \gamma  }$) if $i_1 \neq j_1$, then
$ \sup \{ i_1 , j_1 \} $ has not the same parity of $n$, and 

(b$_{ \gamma  }$) if $i_2 \neq j_2$, then
$ \sup \{ i_2, j_2 \} $ is even.

In passing, let us mention that there is an alternative construction of 
$\mathbf B$ which makes  the definitions of $\beta$ and $\gamma$ 
 simpler; actually, in this alternative
construction $\beta$ and $\gamma$ turn out
to be  kernels of appropriate projections.
See    \cite{baker}.
However, all the remaining arguments are much more involved
in  \cite{baker}; moreover, the current presentation 
has the advantage of being  more compact. 

With the above definitions
of $\alpha$, $\beta$ and $\gamma$, we have $c_0 \mathrel \alpha  c_{n}$
and 
 $c_j \mathrel \alpha  c_{j+1}$, for
$j=1, \dots, n-2 $. Moreover,
 $c_{2i} = (n-2i,2i, - ) \mathrel \beta (n-2i-1,2i+1, - ) =  c_{2i+1}$, 
for all the appropriate values of $i$ and where the value of the third
component is not relevant. Similarly,  
$c_{2i+1} \mathrel \gamma    c_{2i+2}$,
for every  appropriate $i$. 
Hence 
$(c_0, c _{n} ) \in 
\alpha (\beta \circ (\alpha \gamma \circ _{n-2} 
 \alpha \beta) \circ  \gamma ) $, for 
$n$ even, and
$(c_0, c _{n} ) \in  \alpha (\beta \circ (\alpha \gamma \circ
 \alpha \beta \circ  {\stackrel{n-2}{\dots}} \circ \alpha \gamma) \circ  \beta  )$, 
 for $n$ odd. 

On the other hand, 
in view of the above description of
$B ^{ \uparrow } $, 
the only elements
$\alpha \beta $-connected to 
$c_0=e_0 = (n, 0, \uparrow )$ 
are $c_0$ itself and 
 $e_1= (n-1, 0, \uparrow )$.
No other element 
of $B ^{ \uparrow } $
is $\alpha \gamma  $-connected to 
$c_0$, hence there is no advantage in
``staying at $c_0$''. 
The only other element
$\alpha \gamma  $-connected to 
$e_1=(n-1, 0, \uparrow )$ is $e_2=(n-2, 0, \uparrow )$ and, so on, 
the only element 
$\alpha \beta $-connected to 
$e_{2i}$ is $e_{2i+1}$
and the only element 
$\alpha \gamma  $-connected to 
$e_{2i+1}$ is $e_{2i+2}$,
until we reach $e_{n-1}$,
where the situation splits into two
cases.

If $n$ is even, then
$  (1, 0, \uparrow ) = e_{n-1} \mathrel { \alpha \gamma } e_n
= (0, 0, \uparrow ) = f_0 \mathrel { \alpha \beta } f_1 =  (0, 1, \uparrow )$
 and no other nontrivial  relation
holds among these elements.
Symmetrical considerations hold for the
$f_j$'s and, since 
$f _{n}= c _{n} $, 
 we get that 
any chain
from $c_0$ to $c _{n} $  
 in which each pair of elements
is either $\alpha \beta $ or $ \alpha \gamma $-connected
  must involve all the $2n+1$ elements 
of  $B ^{ \uparrow } $,
hence any chain as above is of length
at least $2n$, thus 
\eqref{3} fails in $\mathbf B$.

On the other hand, if $n$ is odd, then
$  (1, 0, \uparrow ) = e_{n-1} \mathrel { \alpha \beta  } e_n
= (0, 0, \uparrow ) = f_0 \mathrel { \alpha \beta } f_1 =  (0, 1, \uparrow )$,
thus $e_{n-1} \mathrel { \alpha \beta  }  f_1 $ 
and we do not need all the
elements of $B ^{ \uparrow } $
to get an $\alpha \beta $-or-$\alpha \gamma $-chain, 
we can skip $e _{n} =  f_0$.
 However, all the rest is the same and we need
$2n-1$ steps from  
$c_0 $ to $ c _{n} $, hence
\eqref{6} fails.
\end{proof}
 
The case $n=2$ 
for $\mathcal B$ 
in Theorem \ref{bds} 
gives another proof of Baker result that 
$\mathcal B$ is $4$-distributive but not 
$3$-distributive. The proof of    
 $4$-distributivity is the same. The counterexample
to $3$-distributivity 
in \cite{B} has 10 elements and
the counterexample 
here can be taken to have
$9$ elements,
since two elements can be discarded from $B$, still 
having an algebra in $\mathcal {B}$, as we shall show
in the proof of  Proposition \ref{propcon} below. 
In the special case  $n=2$ the treatment
from \cite{baker} would be slightly simpler;
 the classes of congruences in 
the example from \cite{B} are to be computed by hand,
while in \cite{baker}  we have considered kernels of projections, which
are automatically congruences. 

\begin{remarks} \labbel{rem}    
(a) There is a short and simple syntactical folklore proof
that Baker's variety is not $2$-distributive, that is, that 
$\mathcal B$  has no
majority term. Actually, the proof shows that $\mathcal B$  
has no near-unanimity term.
If $t$ is a term of $\mathcal B$, define the \emph{relevant}
variable of $t$ inductively as follows. If $t$ is a variable
$x_j$, then $x_j$ is the relevant variable of $t$. Otherwise,
$t=b(t_1, t_2,t_3)$ and we define the relevant variable of $t$
to be the relevant variable of $t_1$. If $\mathbf B \in \mathcal B$,
$\mathbf B$ has a minimal element $0$ and we substitute $0$ for the
relevant variable of some term $t$, then $t$ is evaluated as $0$,
 no matter what we substitute for the other
variables. Thus $\mathcal B$ has no 
near-unanimity term, in particular, no majority term.  

More generally, the argument shows 
that, for every $k$-ary term $t$, there is some 
``place'' $i \leq k$  
such that
$\mathcal B$ satisfies no equation
of the form $t(\dots, y, \dots) =x$,
where $y$ is put in place $i$, $x$
is a variable distinct from $y$ and the other
arguments of $t$ are arbitrary variables.  
This shows that $\mathcal B$
has no \emph{cube term}, as introduced 
by Berman, Idziak, Markovi\'c, 
  McKenzie,  Valeriote and Willard \cite{bimm}.

(b)
Using a different method,
Mitschke \cite{M}  proved that the variety
$\mathcal I$  of implication algebras has
no near-unanimity term. 
Since Baker's variety $\mathcal B$ is interpretable in 
$\mathcal I$, then Mitschke's result furnishes another proof that 
$\mathcal B$ has no near-unanimity term.
The  method in (a) can be applied also to $\mathcal I$,
providing a shorter proof of the mentioned result by  Mitschke.
Simply argue as in (a) above, defining the relevant variable
of $t_1 {\rightarrow} t_2$ to be the relevant variable of 
$t_2$ and dealing with some maximal element $1$ rather than with $0$.   
Thus we also get that $\mathcal I$ has no cube term.
Essentially, this is the argument  hinted on \cite[p.\ 1470]{bimm}. 
In particular, $3$-distributive $3$-permutable varieties 
 do not necessarily have a cube term.

(c) The argument in (a) can be extended in order to give still another proof
that Baker's variety is not $3$-distributive.
Indeed, $3$-distributivity is equivalent to the existence of ternary terms 
$j_1$ and $j_2$ satisfying
$x=j_1 (x,x,y) = j_1 (x,y,x)$, 
$j_1 (x,y,y) = j_2 (x,y,y)$ and
$j_2 (x,x,y)= j_2 (y,x,y)=y$ \cite{J}.
With the same assumptions and definitions as in (a) above, 
the first equations imply that the relevant variable of 
$ j_1 (x,y,z)$ is $x$, hence
$0 = j_1 (0,b,b)= j_2 (0,b,b)$,
for every $b \in B$.
Under the order induced by the 
semilattice operation, we have that every term operation is monotone
(this applies to $\mathcal B$ but \emph{not} to the
variety of implication algebras!),
hence 
$0 = j_2 (0,b,b) \geq j_2 (0,0,b) = b$,
which is impossible if $\mathbf B$
is taken to be of cardinality $ \geq 2$.
 Thus $\mathcal B$  is not $3$-distributive.

In fact, in the above argument we have not  used the equation
$j_2 (y,x,y)=y$.
This shows that $\mathcal B$ 
does not even satisfy
$\alpha( \gamma  \circ \beta  ) \subseteq 
\alpha \gamma  \circ  \beta  \circ \gamma   $,
equivalently, taking converses,
$\mathcal B$ 
does not satisfy $ \alpha ( \beta \circ \gamma )
\subseteq  \gamma \circ \beta  \circ \alpha \gamma $.
This negative result shall be improved in 
the following Proposition.
Compare equation \eqref{7a} below.
\end{remarks}

 In the terminology from \cite{jds}, Theorem \ref{bds} implies  that  
$ J _{ \mathcal B} (n -1 ) = 2n  -1 $, for $n$ even and that
$ J _{ \mathcal B} ( n -1 ) = 2n -2 $, for $n$ odd.
In \cite{jds} we have also considered ``reversed''
J{\'o}nsson spectra, given by identities like    
$\alpha (\beta \circ_n \gamma  )
\subseteq 
\alpha \gamma  \circ _{k} \alpha \beta  $.
We are 
going to see that the proof of Theorem \ref{bds}  gives
 exact bounds 
  for identities
of the above kind both in $\mathcal B$ and $\mathcal N_4$,
as well as in their distributive counterparts.

Moreover, it follows from results by Tschantz \cite{T}
that, for every  congruence modular variety $\mathcal V$ and every $n$, 
there is some $k$ such that $\mathcal V$  
satisfies $\alpha( \beta \circ _n  \gamma ) \subseteq \alpha ( \gamma \circ \beta  )
\circ ( \alpha \gamma  \circ_{k}  \alpha \beta ) $.
See, e.~g., \cite[Section 4]{jds} for details. 
Of course, in a congruence distributive variety we already 
know that $\alpha( \beta \circ _n  \gamma ) \subseteq 
\alpha \beta \circ_{k'}  \alpha \gamma  $, for some $k'$.
However, in principle, it might happen that Tschantz-like formulae 
provide a value of $k$
much smaller than $k'$. 
This is not the case for $\mathcal B$ and $\mathcal N_4$.

\begin{proposition} \labbel{propcon} 
Suppose that  $n \geq 2$
and  $\mathcal V $ is either 
$ \mathcal B $, $ \mathcal B ^d $, $  \mathcal N_4$
or  $\mathcal N_4^d$. 
Then $\mathcal V$ satisfies the following identities.
\begin{align}
\labbel {7} 
\alpha (\beta \circ_n \gamma  )
&\subseteq 
\alpha \gamma  \circ _{2n+1} \alpha \beta     && \text{for $n$ even,}   
\\
\labbel {8}
\alpha (\beta \circ_n \gamma  )
&\subseteq 
\alpha \gamma  \circ _{2n} \alpha \beta     && \text{for $n$ odd,}   
\\ 
\labbel {7a} 
\alpha (\beta \circ_n \gamma  )
&\subseteq 
\alpha (\gamma \circ \beta ) \circ
 ( \alpha \gamma  \circ _{2n-1} \alpha \beta )    && \text{for $n$ even,}   
\\
\labbel {8a} 
\alpha (\beta \circ_n \gamma  )
&\subseteq 
\alpha (\gamma \circ \beta ) \circ
 ( \alpha \gamma  \circ _{2n-2} \alpha \beta )    && \text{for $n$ odd,}   
\end{align}    
and the values of the indices  on  the right-hand sides
give the best possible bounds. 
Actually, $\mathcal V$  fails to satisfy
\begin{align*}
\alpha (\beta \circ (\alpha \gamma \circ _{n-2} 
 \alpha \beta) \circ  \gamma ) 
&\subseteq 
\alpha(\gamma \circ \beta ) \circ
 ( \alpha \gamma  \circ _{2n-4} \alpha \beta ) \circ  
\alpha(\gamma \circ \beta ), \text{ for $n$ even,}  
\\
\alpha (\beta \circ (\alpha \gamma \circ_{n-2} \alpha \beta )
\circ  \beta  ) 
&\subseteq 
\alpha (\gamma \circ \beta ) \circ
(\alpha \gamma \circ _{2n-5}
 \alpha \beta )
  \circ
\alpha ( \beta \circ \gamma ),
\text{ for $n$ odd.}
\end{align*}
\end{proposition} 

\begin{proof}
Equations \eqref{7} - \eqref{8a} are immediate from \eqref{1} and \eqref{4}, since, say,
$\alpha \beta \circ _{2n} \alpha \gamma   \subseteq
\alpha \gamma  \circ _{2n+1} \alpha \beta $
and $ \alpha \beta \subseteq \alpha ( \gamma \circ \beta )$.

The proof of Theorem \ref{bds} shows
that the bounds on the right-hand sides
of \eqref{7} and  \eqref{8}
 are optimal.
Indeed, in the proof that \eqref{3} and  \eqref{6} fail
we have observed that $c_0$ is $\alpha \gamma $-connected to 
no other element of $B ^{ \uparrow } $, hence we ``lose one turn''
if we want the chain to start with  $\alpha \gamma $.
Actually, we have that, say, for $n$ even, already
$\alpha (\beta \circ (\alpha \gamma \circ _{n-2} 
 \alpha \beta) \circ  \gamma) \subseteq 
\alpha \gamma  \circ _{2n} \alpha \beta$
fails in $\mathcal V$. 

In order to show that the indices in \eqref{7a} and   \eqref{8a}
are best possible,
we shall modify the construction in 
the proof of Theorem \ref{bds}.
With the definitions and 
notations in the mentioned proof,
let $B^- = B \setminus \{ (n,0, \downarrow), (0,n,\downarrow) \} $.
We claim that $B^-$ is (the base set  for)
an algebra in, say,  $\mathcal {N}_4$.   
We shall show that if 
$a_1,a_2,a_3,a_4 \in B^-$, then 
it is not the case  that 
$u(a_1,a_2,a_3,a_4) =(n,0, \downarrow)$.
Indeed, since $c_0$ is the only element of 
$B^-$ with first component $n$,  if the first 
component of  $u(a_1,a_2,a_3,a_4)$ is $n$, then 
at least three arguments of $u$ have $n$ as the first component,
hence at least three arguments of $u$ are equal to 
$c_0$, thus  $u(a_1,a_2,a_3,a_4)$ is itself  $c_0$. 
 Notice that, by construction, $n$ is the maximum possible value
for the first component.  
Similarly, if 
$a_1,a_2,a_3,a_4 \in B^-$, then 
$u(a_1,a_2,a_3,a_4) $ is not $ (0,n, \downarrow)$,
hence  $\mathbf B^-$ is an algebra in $\mathcal {N}_4$. 

We have that $c_0$ is $ \gamma $-connected to 
no other element of  $B^-$,
because of clause 
(a$_ \gamma $) in the definition of $\gamma$.
Thus if  
$c_0 \mathrel { \alpha ( \gamma \circ \beta )} f $ in $B^-$,
for some $f$, then
$c_0 \mathrel {  \gamma } e \mathrel { \beta  }f $,
for some $e$, hence necessarily $c_0=e$ and 
$c_0 \mathrel { \alpha \beta } f $.
Thus if, say,  $n$ is even and
we suppose by contradiction that
$(c_0,c_{n}) \in \alpha (\gamma \circ \beta ) \circ
 ( \alpha \gamma  \circ _{2n-2} \alpha \beta ) $,    
then we would have 
$(c_0,c_{n}) \in
  \alpha \beta   \circ _{2n-1} \alpha \gamma   $,
but this is impossible
because of the counterexample constructed in the proof of Theorem \ref{bds}.  
Hence  we cannot  get better bounds in  \eqref{7a} or  \eqref{8a}.
Performing also the symmetric argument, we have that $\mathcal V$
fails to satisfy the last equations in the statement. 
\end{proof}
 
Recall that a variety $\mathcal V$ is \emph{$n$-modular} 
if $\mathcal V$ satisfies the identity
$\alpha (\beta \circ \alpha  \gamma \circ \beta  )
\subseteq 
\alpha \beta   \circ _{n} \alpha \gamma $.  
 Cf. Day \cite{D}.  Equations \eqref{4} and   \eqref{6} 
in Theorem \ref{bds}  in the case 
$n=3$  provide the following corollary.  

\begin{corollary} \labbel{dm}
The varieties  
$ \mathcal B $, $ \mathcal B ^d $, $  \mathcal N_4$
and  $\mathcal N_4^d$
  are  $4$-distributive,
 $5$-modular and  not 
$4$-modular.  
 \end{corollary}

\section{Some relation identities} \labbel{relid} 

We shall  use  $R$, $S$  and $T$ as variables for 
reflexive and admissible binary relations
and $\Theta$ as a variable for tolerances.
All the relations considered in the present note 
are assumed to be reflexive, hence we shall sometimes simply
say \emph{admissible} in place of \emph{reflexive and admissible}. 
In \cite[Proposition 3.1]{jds} we  noticed that
congruence distributive varieties satisfy 
also
 relation identities
of the form
$ \Theta  (R \circ_n S  )
\subseteq 
\Theta R \circ _{k} \Theta S $.
This is  a consequence of 
results by Kazda,  Kozik,   McKenzie and Moore \cite{adjt}. 
See 
 \cite{jds} for further details and references.

We do not know whether every congruence distributive variety 
satisfies $ T (R \circ S  )
\subseteq 
T R \circ _{k} T S $, for some $k$.
However, 
we showed in \cite{uar} that the above relation
holds in Baker's variety with $k=4$.
In the present section we shall provide 
 exact bounds 
  for identities
of the above kind,
both in the case of $\mathcal B$ 
and in the case of $\mathcal N_4$.  
We shall exhibit a subtle difference between 
the two varieties,
when relation identities are concerned.
Compare Theorems \ref{pari}, \ref{proprelb} and
  Proposition \ref{moregen}
 below.
On the other hand, as we showed
in the previous section, $\mathcal B$ 
and $\mathcal N_4$ 
behave in the same way as far as
 congruence identities are concerned.

Each result in the present and in the following section
holds for $\mathcal {B}$ if and only if it holds for 
$\mathcal {B}^d$. In fact, we shall never use lattice distributivity 
in the proofs;
on the other hand, all the counterexamples we shall deal with
are based on the construction in the proof of Theorem \ref{bds}
and this construction is
the reduct of a distributive lattice. 
Similarly, each result holds for $\mathcal {N}_4$ if and only if it holds for 
$\mathcal {N}_4^d$.
For the sake of simplicity, we shall not mention
$\mathcal {B}^d$ and $\mathcal {N}_4^d$ explicitly in the following statements.
However,  the reader might always consider 
$\mathcal {B}^d$ in place of  $\mathcal {B}$ and $\mathcal {N}_4^d$ 
in place of  $\mathcal {N}_4$ in what follows.

Recall that, if $n$ is odd, we sometimes  write 
$ R \circ S \circ {\stackrel{n}{\dots}} \circ R$
in place of $R \circ _n S$,
when we want  to make clear 
that the last factor is $R$.

\begin{theorem} \labbel{pari}
If $n \geq 2$, then the following identities are satisfied:
\begin{align}
\labbel {9}
T (R \circ_n S)& \subseteq  T R  \circ _{2n} T S,
 \text{\ \ \ \  by $\mathcal B$, $\mathcal N_4$, $n$ even;}  
\\
\labbel {10}
 T (R \circ_n S)& \subseteq  
 (T R  \circ T S  \circ  {{\stackrel{n}{...}}}  \circ  TR ) \circ
 (T R  \circ T S \circ {{\stackrel{n}{...}}} \circ  TR ), 
 \text{\  by $\mathcal B$, $n$ odd;}  
\\
\labbel {10bn}
 T (R \circ_n S) & \subseteq  
T R  \circ _{2n-1}  T S ,
\text{\ \ by $ \mathcal N_4$, $n$ odd,}    
 \end{align}    
and the bounds  on  the right
are best possible; moreover, $\mathcal B$
fails to satisfy 
\begin{equation}\labbel{10b}     
\alpha  ( \Theta  \circ_n \gamma ) \subseteq  
\alpha  \Theta   \circ _{2n} \alpha  \gamma   , 
  \end{equation}
for $n$ odd, where $\Theta$ is a tolerance and $\alpha$ and $\gamma$ 
are congruences.
 \end{theorem} 

Before proving Theorem \ref{pari} we shall prove
Proposition  \ref{moregen} below,
 a  more general result
whose formulation, however, is  more involved.
The statement
of Theorem \ref{pari}
suggests that, for 
 $n$ even, $\mathcal B$ and $\mathcal N_4$ behave
 in the same way
and  no essential difference
seems to appear 
in comparison with the previous section.
On the other hand, for $n$ odd,  Theorem   \ref{pari} shows 
that $\mathcal N_4$ satisfies the somewhat stronger identity
 \eqref{10bn}. However, in a sense,
 $\mathcal N_4$  behaves better than $\mathcal B$
for every value of the index $n$, as we shall show in 
  Proposition  \ref{moregen}.
The difference between $\mathcal B$
and $\mathcal N_4$ 
will appear
in a clearer light in Theorem \ref{proprelb}.

After the proof of
Proposition  \ref{moregen},
furnishing the positive side of Theorem \ref{pari},
 we shall present the example of a tolerance
on $(B,b)$, 
 the algebra constructed in the proof of Theorem \ref{bds}. 
The example  shall then be used 
in order to show that 
the bounds in  identity \eqref{10}  are optimal.

For a binary relation $R$, let $R^n$ denote the 
$n$-fold composition of $R$ with itself, that is, 
  $R^n = R \circ_{n} R$.
If $R$ and $S$ are binary relations, let
$ \overline{R  \cup S} $ denote the smallest  
 admissible relation containing both 
$R$ and $ S$.

\begin{proposition} \labbel{moregen}
For every  $n \geq 2$, the  
following identities are satisfied:
\begin{align}\labbel{nmg} 
\text{by $\mathcal B$: \ \ \ } & T  (R_1 \circ  R_2 \circ \dotsc  \circ R_n)  \subseteq 
(T  R_1 \circ T   R_2 \circ \dotsc   \circ
T  R_n )^2 ;
\\
\labbel{bmg} 
\begin{split}
\text{by  $\mathcal N_4$: \ \ }         
 & T  (R_1 \circ  R_2 \circ \dotsc \circ  R _{n-1} \circ R_n)
  \subseteq 
T  R_1 \circ T   R_2 \circ \dotsc \circ T  R _{n-1} \circ 
\\
 & \qquad \quad T (TR_n  \circ TR_1)(\overline{R_n  \cup R_1}) 
\circ T   R_2 \circ \dotsc \circ T  R _{n-1} \circ  TR_n .
\end{split}
\end{align}   
 \end{proposition}

\begin{proof}
Were $T$ a congruence, the proof of \eqref{1} 
in Theorem \ref{bds} 
would give a proof of  \eqref{nmg}, 
since in the proof of Theorem \ref{bds}
we have not used the assumption that $\beta$ and $\gamma$
 are congruences, we have
only used that $\beta$ and $\gamma$  are 
 admissible relations.
Dealing with many relations instead of just two relations
presents no new difficulty.
At certain places  in Theorem \ref{bds} we need transitivity
of $\beta$, but this is not necessary here, according the formulations 
of Theorem \ref{pari} and Proposition \ref{moregen}.
For example, were both $T$ and $R$  transitive in 
equation \eqref{10}, then the two adjacent  
  occurrences of $TR$ would absorb into one.
This is the reason why the indices in Theorem \ref{bds}
can be improved by one in the case $n$ odd,
since $\alpha$ and $\beta$ there are congruences, hence transitive.    
But we are not assuming transitivity in Theorem \ref{pari}
and in Proposition \ref{moregen}; correspondingly, we are not
asking for the stronger conclusions. 

We have to give a proof of
Proposition \ref{moregen}  in the case when
$T$ is just an admissible relation.
This involves considering a
 different sequence of elements
in comparison with the sequence described in \eqref{B}. 
We shall first  give the proof 
of \eqref{nmg} 
 for $\mathcal B$;
then we shall improve
\eqref{nmg}
to 
\eqref{bmg}
for 
 $\mathcal N_4$
by an additional 
and somewhat delicate argument.

Suppose that 
$(a,d) \in T  (R_1 \circ  R_2 \circ \dotsc  \circ R_n)$
in some algebra in $\mathcal B$.
This means that
$a \mathrel T d$ and 
that there are elements $c_i$ such that  
 $a= c_0 \mathrel { R_1 } c_1 \mathrel { R_2} c_2 
\dots
\mathrel { R_{n-1} } 
c _{n-1} \mathrel { R_n } c_n =d  $.
In what follows, we shall
write, say,
$a(d+c_1)$
for 
$b(a,d,c_1) $, and
$a(d+c_1)(d+c_2)$
for 
$b(b(a,d,c_1),d,c_2) $.
Equations like
$a(d+c_1)(d+c_1) = a(d+c_1)$
or 
$a(d+c_1)(a+c_2) = a(d+c_1)$
hold in $\mathcal B$ 
since corresponding equations hold
in lattices.
 Consider the following elements.

\begin{align*}       
g_0&=a                           & 
 g_1&= a(d+c_1) \\
 g_2&=  a(d+c_1)(d+c_2)   &  
 &\dots \\
 g_{n-1}&=  a(d+c_1)(d+c_2)\dots(d+c_{n-1})                             &  
 g_n &=h_0= ad \\
 h_{1}&=  d(a+c_1)(a+c_2)\dots(a+c_{n-1})                         &  
 h_{2}&=  d(a+c_2)\dots(a+c_{n-1})\\
&\dots                            
\quad \quad  h_{n-1}=  d(a+c_{n-1})     &    
 h_{n}&=  d
\end{align*} 

We have $g_i \mathrel { TR _{i+1}} g _{i+1}  $ for $i<n$ and similarly for the
$h_i$'s.  Indeed, for example, 
\begin{align*}
 g_1= a(d+c_1)= a(d+c_1)(a+c_2) &\mathrel { \hskip 2.5pt T \hskip 2.5pt }   
a(d+c_1)(d+c_2) =  g_2, \text{ and}
\\  
 g_1= a(d+c_1)= a(d+c_1)(d+c_1) &\mathrel { R_2}   
a(d+c_1)(d+c_2) =  g_2.
 \end{align*}
Notice that, in the definition of the 
$g_i$'s, when going from 
 $g_{n-1}$ to $g_n$
we follow the preceding pattern; indeed, according to the pattern,
$g_n$ would be
$a(d+c_1)\dots(d+c_{n-1})(d+c_n) = 
a(d+c_1)\dots(d+c_{n-1}) \allowbreak (d+d) $ which in fact is equal to 
$ ad$.   Thus 
we have 
$ (a,g_n) \in T  R_1 \circ T   R_2 \circ \dotsc   \circ
T  R_n $ and
$ (h_0,d) \in T  R_1 \circ T   R_2 \circ \dotsc   \circ
T  R_n $, hence 
 $(a,d) \in (T  R_1 \circ T   R_2 \circ \dotsc   \circ
T  R_n ) ^2   $,
since $g_n= h_0$.

We have proved \eqref{nmg}. 

In passing, let us remark that the above elements 
$g_0, g_1, \dots, g_n=h_0,  \allowbreak \dots, h _{n-1}, h_n $
 satisfy some additional relation identities.
We have that each element in the above chain
is $T$-related with every element which follows. For example,
 $g _{n-1} \mathrel { T } h_1 $,
since
\begin{multline*}
 g_{n-1}=   a(d+c_1)\dots(d+c_{n-1}) = 
a(d+c_1)\dots(d+c_{n-1}) (a+c_1)\dots(a+c_{n-1}) 
\mathrel { T } 
\\
d(d+c_1)\dots(d+c_{n-1}) (a+c_1)\dots(a+c_{n-1}) 
=  d(a+c_1)\dots(a+c_{n-1}) = h_{1}.
 \end{multline*}   

All the other relations are proved in a similar way.

The proof of equation \eqref{bmg}
involves the same chain of elements,
this time working in 
$\mathcal N_4$. 
Notice that the above elements are expressible in $\mathcal N_4$,
since   $b(x_1, x_3, x_4) = u(x_1, x_1, x_3, x_4)= 
x_1(x_3+x_4)$.
In the above-displayed formula we have  showed 
$g_{n-1} \mathrel {  T }  h_1$.
It remains  to show that, in addition,
$g_{n-1} \mathrel { \overline{R_n \cup R_1} } h_1$. 
 In order to prove this relation,
we shall write 
$g_{n-1}=   a(d+c_1)\dots(d+c_{n-1}) $ 
as
$u(a,a,d,c _{n-1} ) \cdot (d+c_1)\dots(d+c_{n-2}) $.
This formula should be interpreted 
in the sense that, say,
$u(a,a,d,c _{n-1} ) \cdot (d+c_1)$
is an abbreviation for
$b(u(a,a,d,c _{n-1} ), d, c_1)$,
and we can add further factors 
of the form 
$ (d+c_j)$ by  iterated applications of $b$,
as we did in the first part of the proof.   
The identities we shall use will all follow
from corresponding identities holding in lattices.
Now we compute
\begin{align*} 
g_{n-1}= a(d+c_1) &\dots(d+c_{n-2}) (d+c_{n-1})
= \\
a(d+c_{n-1})(d+c_1) &\dots(d+c_{n-2}) 
(a+c_2)\dots(a+c_{n-1}) =
\\ 
u(a,a,d,c _{n-1} ) \cdot (d+c_1)&\dots(d+c_{n-2}) 
(a+c_2)\dots(a+c_{n-1})  \mathrel { \overline{R_n \cup R_1} }
\\
u(a,c_1,d,d) \cdot (d+c_1)&\dots(d+c_{n-2}) 
(a+c_2)\dots(a+c_{n-1})  =
\\
d(a+c_1)(d+c_1)&\dots(d+c_{n-2}) 
(a+c_2)\dots(a+c_{n-1}) = 
\\
&\quad \ \ \ \, d(a+c_1)
(a+c_2)\dots (a+c_{n-1}) = h_1.
\end{align*}   
Thus 
$g _{n-1} \mathrel { \overline{R_n \cup R_1} } h_1 $ 
and  the proof of 
Proposition \ref{moregen} is complete. 
\end{proof}

Recall the definitions of $ \alpha $, $\beta$,  $ \gamma $, 
$B$, $B ^ \uparrow $
from  the proof of Theorem \ref{bds}.
Let  $B ^ \downarrow = B \setminus B ^ \uparrow$
be the set of those elements of $B$ 
with a third $ \downarrow $ component. Let
$E= \{ e_0, \dots  e_n\} $ and $F= \{ f_0, \dots  f_n\} $,
where  $e_0, \dots , f_n$ are the elements  in
the displayed list
\eqref{bu} in the proof of Theorem \ref{bds}.  
 We now present the example of a tolerance
on $(B,b)$.
Recall that $(B, b)$ is an algebra in $\mathcal B$. 

\begin{example} \labbel{ex1} 
Let $ \Lambda  $  be the binary relation on $B$ defined
as follows.
 Two elements $x$ and $y$ of $B$ are $ \Lambda $-related if and only if 
either
 
(a) both $x $ and $ y $ belong to $  E$, or 

(b) both $x $ and $ y $ belong to $  F$, or

(c) at least one of $x$ and  $y$ 
belongs to $B ^ \downarrow$. 

We are going to show that $\Lambda$ is a tolerance on $(B, b)$. 

Indeed,  $\Lambda$  
is clearly  symmetric and reflexive,
since $B ^ \uparrow =  E \cup F$,
hence  $B ^ \downarrow = B \setminus ( E \cup F)$.
We have to check that $\Lambda$   is admissible.
Suppose that $x_1 \mathrel \Lambda y_1 $, $x_2 \mathrel \Lambda  y_2 $
and $x_3 \mathrel \Lambda y_3 $ are elements of $B$. 
Letting  $x= b(x_1,x_2, x_3)$
and
 $y= b(y_1,y_2, y_3)$,
we have to show that
$x \mathrel \Lambda y$. 
If either 
$x \in B ^ \downarrow$ 
or $y \in B ^ \downarrow$, there
is nothing to prove. 
If  both
  $x= b(x_1,x_2, x_3)$
and
 $y= b(y_1,y_2, y_3)$
belong to $B ^ \uparrow $,
that is, they have a last $ \uparrow $
component, then both
$x_1$ and  $y_1$   
have a last $ \uparrow $
component, that is, 
$x_1, y_1 \in B ^ \uparrow$.   
By the definition of $\Lambda$, 
either 
$ x_1, y_1 \in E$
or $x_1, y_1\in F$, and, correspondingly, 
$ x_1 $ and $  y_1 $
either  both have a null second component
or  have a null first component;
hence this applies to $x$ and $y$, too.
By the description of $B ^ \uparrow $
in the proof of Theorem \ref{bds},
if $x$ and $y$ have a null second component, then they both
belong to $E$, and 
if $x$ and $y$ have a null first component, then they both
belong to $F$.
We have showed that 
$ x \mathrel \Lambda y$, thus $\Lambda$ is admissible
(we have not used the assumption that  $x_2 \mathrel \Lambda  y_2 $
and $x_3 \mathrel \Lambda y_3 $).

On the other hand, $\Lambda$ is not admissible on $(B, u)$;
see Remark \ref{lr}.
\end{example}

 \begin{proof}[Proof of Theorem \ref{pari}] 
The positive result that equations \eqref{9} - \eqref{10bn}
hold is an immediate consequence of Proposition \ref{moregen},
taking   $R_1= R_3= \dots = R$ 
and $R_2= R_4= \dots = S$.
Notice that if $n$ is odd, then 
$R_n  =  R_1 = R$, hence the factor 
$T (TR_n  \circ TR_1)(\overline{R_n  \cup R_1})$ 
in \eqref{bmg} 
becomes 
$TR$,
thus we get  
\eqref{10bn} from  \eqref{bmg}.

The bounds given by
equations
\eqref{9} and \eqref{10bn} are
optimal
even in the case of congruences,
as shown by equations \eqref{4} and  \eqref{6}
in Theorem \ref{bds}. 
As soon as we show that
\eqref{10b} fails, we get that 
\eqref{10} is the best possible result; in particular, 
 the two
adjacent
occurrences of  $TR$ in the middle of  the right-hand side of 
 \eqref{10} do not always  ``absorb into one'',
even in the case when $T$ is a congruence and
$R$ is a tolerance.

To show that \eqref{10b} can fail, consider the construction in the proof
 of Theorem \ref{bds}  in the case $n$ odd, this time taking
 $\mathbf B= (B, b) $
in place of 
 $\mathbf B= (B, u) $.
Let $\Theta =  \Lambda  \beta$,
where $\Lambda$ is the tolerance constructed in
Example \ref{ex1}.  
Proceed as in the proof of Theorem \ref{bds}  
and notice that
$(c_0,c_n) \in \alpha  ( \Theta  \circ_n \gamma )$,  
as witnessed, again, by $c_1, c_2, \dots$\  In the case at hand
$e_{n-1}= (1,0, \uparrow )$ and $f_1= (0,1, \uparrow )$
 are not $\Theta$-related, 
since they are not
$\Lambda$-related. Hence here we cannot
skip the passage from 
$e_{n-1}$ to $e_{n}$, as we did
 in the case $n$ odd in the proof of Theorem \ref{bds}.
Of course, we do have
$(c_0,c_n) \in 
( \alpha  \Theta  \circ \alpha \gamma \circ {\stackrel{n}{\dots}} \circ \alpha  \Theta   )
\circ
( \alpha  \Theta  \circ \alpha \gamma \circ {\stackrel{n}{\dots}} \circ \alpha  \Theta   )$,
 as shown be equation \eqref{10}; 
however, we lose one more turn if we want that 
$\alpha  \Theta   $ and $  \alpha  \gamma $ strictly alternate
on the right-hand side of \eqref{10b},
hence we cannot have
$(c_0,c_n) \in 
\alpha  \Theta  \circ_{2n} \alpha \gamma$.
 \end{proof}  

Notice that the above argument  shows that
even the identity
$ \alpha (\Theta  \circ 
(\alpha \gamma  \circ \alpha \Theta  \circ {\stackrel{n-2} {\dots}} \circ 
\alpha \gamma ) \circ  \Theta ) \subseteq  
\alpha \Theta   \circ _{2n} \alpha \gamma    $
 fails in $\mathcal B$, for $n$ odd.

\begin{remark} \labbel{contol!} 
In \cite{contol} we have showed that, under a fairly general hypothesis,
any congruence identity is equivalent to the corresponding tolerance
identity, provided that only tolerances \emph{representable} 
as $R \circ R ^\smallsmile $ are considered in the latter identity.
Here $R ^\smallsmile $ denotes the \emph{converse}
of $R$. 
 
Equations \eqref{4} in Theorem \ref{bds}
and \eqref{10b} in Theorem   \ref{pari}
show that, in general, the assumption of representability 
is necessary in the results from \cite{contol}.
A similar counterexample has been presented in
\cite{ia}.

It  follows from \cite{contol}
that the tolerance $\Theta$ used in the proof of Theorem \ref{pari}
is not representable.
It can be checked directly that in $(B,b)$  neither $\Theta$ for $n$ odd,
nor $\Lambda$    for every $n$ are representable, where $\Lambda$ 
is the tolerance constructed in Example \ref{ex1}. 
In fact, if $R$ is reflexive and admissible,
$c_0 \mathrel { R \circ R ^\smallsmile }  c_1$ and  
$c_{n-1} \mathrel { R \circ R ^\smallsmile }  c_n$,
then
$e _{n-1} \mathrel { R \circ R ^\smallsmile }  f_1$.  
Indeed, if 
$c_0 \mathrel { R } g \mathrel { R ^\smallsmile }  c_1$ and  
$c_{n-1} \mathrel { R } h \mathrel {R ^\smallsmile }  c_n$,
then $c_n \mathrel R h$ and $ g \mathrel { R ^\smallsmile }  c_0$, hence  
$e_{n-1} = c_0 (c_{n-1} + c_n) \mathrel R   g(h + h) = gh =
h(g+g) \mathrel {R ^\smallsmile }  c_n (c_0 + c_1) = f_1 $. 
\end{remark}

\section{Identities  with just two relations} \labbel{two}

If $R$ is a congruence,
or just a transitive relation, then
then there is no point in considering identities of the form
$T ( R \circ R ) \subseteq  something$,
since \mbox{$R \circ R = R$}. In passing, let us point out
that  this latter identity $R \circ R = R$ 
 is equivalent to congruence permutability, as noticed
independently by  Hutchinson \cite{H} and  Werner \cite{W}.
 If we only suppose that $R$ is admissible, 
many more identities of the  form $T ( R \circ R ) \subseteq  something$
become interesting.
For example, we showed in \cite{ricm}
that a variety $\mathcal V$ is congruence modular 
if and only if there is some $k$ such that $\mathcal V$
satisfies the identity 
$\Theta (R \circ R) \subseteq (\Theta R)^k $.

In the present section we evaluate the best possible bounds 
for identities of the above kind both in 
$\mathcal B$ (equivalently, $\mathcal {B}^d$)  and  $\mathcal N_4$
(equivalently  $\mathcal N_4^d$). 
In a certain respect, here the situation 
is simpler than in the preceding sections, since we do not need the division
into the two cases $n$ odd and  $n$ even;
moreover, the bounds for  $\mathcal N_4$ 
are always better than the bounds for 
$\mathcal B$. Notice that all the identities considered
in  the present section are strictly weaker than
congruence distributivity, since they might hold in 
non distributive congruence modular varieties; actually, all the identities below
hold in congruence permutable varieties.

\begin{theorem} \labbel{proprelb} 
For  $n \geq 2$, 
 the  
following identities are satisfied:
\begin{align} \labbel{11}
\text{by $\mathcal B$: \ \ \ } &
 T R^n \subseteq  (T R )^{2n}   
\\
\labbel{11n}
\text{by $\mathcal N_4$: \ \ \ } &
 T R^n \subseteq  (T R )^{2n-1}   
\end{align}    
and the exponents on the right 
are best possible; actually, 
\begin{align}     
\labbel {12ii}
\text{$\mathcal B$
fails to satisfy \ \ \ } 
&\alpha ( \Theta  \circ (\alpha \Theta ) ^{n-2} \circ  \Theta  ) 
\subseteq 
(\alpha \Theta ) ^{2n-1}, \text{  and} 
\\
\labbel {12iin}
\text{$\mathcal N_4$
fails to satisfy \ \ \ } 
&\alpha ( \Theta  \circ (\alpha \Theta ) ^{n-2} \circ  \Theta  ) 
\subseteq 
(\alpha \Theta ) ^{2n-2}.
 \end{align} 
\end{theorem}  

\begin{proof} 
Identities \eqref{11} and \eqref{11n}  are immediate
 from identities \eqref{9} - \eqref{10bn}
in Theorem \ref{pari}, taking
$S=R$; they can also be obtained  directly from Proposition \ref{moregen}. 

We first show that  \eqref{12iin}  can fail,
hence the bound in \eqref{11n} is best possible. 
Consider again the counterexample 
$(B, u)$ constructed in the proof
of Theorem \ref{bds}.
 Let $\Psi$ be the binary relation on $B$
defined in such a way that two elements
$x$ and $y$ in $B$  are $\Psi$-related if and only if

(d) for each $\ell=1,2$, the components
$x_ \ell$ and  $y_ \ell$
differ at most by $1$.

We claim that  $\Psi$ is a tolerance on $\mathbf B$.
Indeed,  condition 
(d) defines a tolerance $\Psi_L$ on the lattice
$\mathbf L=  \mathbf C_{n+1} \times \mathbf C_{n+1} 
\times \mathbf C_2 $,
since $\Psi_L$ is a product of
tolerances on the factors.
Hence   $\Psi_L$ is a tolerance   
on the polynomial reduct 
$(L, u)$ and
 $\Psi$,
being  the restriction of 
 $\Psi_L$ to $B$,
 is a tolerance, too.

Now we have 
$(c_0,c_n) \in \alpha ( \Psi \circ (\alpha \Psi) ^{n-2} \circ  \Psi )$, 
again, 
as witnessed by $c_1, c_2, \dots$\ 
On the other hand, as in the proof of Theorem \ref{bds},
the only other element  $\alpha \Psi$-connected to 
$c_0=e_0$ is $e_1$,
the only other element  $\alpha \Psi$-connected to 
$e_1$ is $e_2$ and so on, until
we reach $e_{n-1}$,
which is   $\alpha \Psi$-connected only to
$e_{n-2}$ (but this has no use), to $e_n=f_0$ and to $f_1$.  
We get the fastest path going directly through
$f_1$; in any case, we need $2n-1$ steps,
thus   
$(c_0,c_n) \in (\alpha \Psi) ^{2n-2} $ fails in 
$\mathcal {N}_4$.
Hence  \eqref{12iin} fails with 
$\Psi$ in place of $\Theta$.

In order to 
disprove 
the identity in \eqref{12ii}, let us work 
in $(B,b)$  instead.
Recall that $(B,b) \in \mathcal {B}$.
The relation $\Psi$
defined above is a tolerance on 
   $(B,b)$, being a tolerance on 
$(B,u)$.
Let $\Theta= \Lambda \Psi$, 
where $\Lambda$ is the tolerance defined 
in Example \ref{ex1}. 
As above,
we have
$(c_0,c_n) \in \alpha ( \Theta  \circ (\alpha \Theta ) ^{n-2} \circ  \Theta  )$.
We have that 
$(c_0,c_n) \in (\alpha \Theta ) ^{2n-1}$ fails,
since any chain of   
$\alpha \Theta$-related elements
from $c_0 $ to  $ c_n $ 
must contain all the 
elements of $B ^ \uparrow $.
The difference 
with the previous case dealing with 
$\mathcal {N}_4$  is that here  
$e _{n-1}  $ 
and  
$f _{1}  $ 
 are not $\Theta$-related, being not
$\Lambda$-related,
hence one more step is necessary.
\end{proof}

\begin{remark} \labbel{lr}
In Example \ref{ex1}
we have seen that $\Lambda$   
is a tolerance on $(B,b)$. 
It follows from the above proof that $\Lambda$ is not a
tolerance on
 $(B,u)$.
It is easy to see  directly that $\Lambda$ is not even compatible
in $(B,u)$; otherwise, we would get
$e _{n-1} = u(c_0,c_0,c_{n-1},c_n)
\mathrel {  \Lambda} u(c_0,c_1,c_{n},c_n) =f_1 $,
a contradiction. 
 \end{remark}

As a small improvement on some results
in this and in the previous section, notice that in the identities in
\eqref{9}, \eqref{10}, \eqref{11} and \eqref{nmg} 
it  is enough to assume that $T$, $R$ and $S$  
are set-theoretical unions of admissible relations.
Indeed, in the proofs
only one element is moved at a time.
Cf.\ \cite{uar}. 

The following lemma provides a simpler argument
to show that the exponent on the right in the identity in  \eqref{11n}
cannot be improved. 
Recall that if $T$ is a binary relation on some algebra,
$ \overline{T} $ denotes the smallest reflexive and admissible relation 
containing $T$. The definition of $n$-modularity has been 
recalled shortly before Corollary \ref{dm}.

\begin{lemma} \labbel{rm}
Let $\mathcal V$ be any variety.
  \begin{enumerate}  
  \item   
 If $\mathcal V$ satisfies
$\alpha(R \circ R) \subseteq (\alpha R) ^{k} $,
then $\mathcal V$ is $2k$-modular.  

\item
 More generally, if $\mathcal V$ satisfies
$\alpha(R \circ_n R) \subseteq (\alpha R) ^{k} $,
for some $n \geq 2$,
then $\mathcal V$ satisfies
$\alpha ( \beta \circ _{2n}  \alpha \gamma )
\subseteq 
\alpha \beta \circ _{2k} \alpha \gamma$. 

  \item   
 If $k \geq 2$ and  $\mathcal V$ satisfies
$\alpha(R \circ S) \subseteq 
\alpha  R \circ (\alpha ( \overline{R \cup S} )) ^{k-1}  $,
then $\mathcal V$ is $2k-1$-modular.  
  \end{enumerate} 
 \end{lemma}

\begin{proof} 
(1)
Taking $R= \beta \circ \alpha \gamma $,
we have
$\alpha (\beta \circ \alpha \gamma)=
\alpha \beta \circ \alpha \gamma$, hence 
\begin{multline*} 
\alpha ( \beta \circ \alpha \gamma \circ \beta ) \subseteq 
 \alpha ( \beta \circ \alpha \gamma \circ \beta  \circ \alpha \gamma )=
\alpha (R \circ R)  \subseteq 
\\
(\alpha R) ^{k} =  (\alpha (\beta \circ \alpha \gamma)) ^{k} =
(\alpha \beta \circ \alpha \gamma) ^{k} =
\alpha \beta \circ _{2k} \alpha \gamma . 
 \end{multline*}   

(2) is proved in the same way.

(3) Take  $R = \beta   $,
 $S= \alpha \gamma \circ \beta $
and observe that
$ R \cup S = \alpha \gamma \circ \beta $
is admissible.  
\end{proof}

Thus $\mathcal N_4$
fails to satisfy 
 $ T R^2 \subseteq  (T R )^{2} $,
since otherwise 
$\mathcal N_4$ would be 
$4$-modular, by \ref{rm}(1), contradicting Corollary \ref{dm}.  

More generally, equation \eqref{11n} cannot be improved to
$ T R^n \subseteq  (T R )^{2n-2} $, 
since otherwise
Lemma \ref{rm}(2)
would give 
\begin{equation}\labbel{blu}    
\alpha (\beta \circ (\alpha \gamma \circ
 \alpha \beta \circ  {\stackrel{2n-3}{\dots}} \circ \alpha \gamma) \circ  \beta  ) 
\subseteq 
\alpha ( \beta \circ _{2n}  \alpha \gamma )
\subseteq ^{\ref{rm}(2)} 
\alpha \beta \circ _{4n-4} \alpha \gamma,
  \end{equation}
where the superscript in 
$\subseteq ^{\ref{rm}(2)} $
means that we are applying 
item (2) in Lemma \ref{rm}. 
Taking $m=2n-1$ in 
\eqref{blu},
we get  $2n-3 =m-2 $
and $4n-4 = 2m -2$,  
thus equation \eqref{blu} becomes  
equation \eqref{6}  in Theorem \ref{bds} with $m$ in place of $n$ 
and Theorem \ref{bds} shows that this equation fails for $\mathcal N_4$.

\begin{remark} \labbel{tolmod}
Recall that $R ^\smallsmile $ denotes the \emph{converse}
of the relation $R$. 
It is not difficult to show that a variety $\mathcal V$ 
is  
$k+1$-modular if and only if
 $\mathcal V$ satisfies the identity
\begin{equation}\labbel{ri}
     \alpha (R \circ R ^\smallsmile ) \subseteq 
\alpha  R \circ_{k} \alpha  R ^\smallsmile .
  \end{equation}
See \cite{ricm}, 
where it is also shown that
 $\alpha$  can be equivalently taken to vary among tolerances.
If we let $R=\Theta$
be a tolerance in \eqref{ri},
we get 
\begin{equation}\labbel{ric}    
 \alpha ( \Theta \circ \Theta )
\subseteq 
( \alpha \Theta) ^{k}. 
  \end{equation}

Clearly, in turn, \eqref{ric}
implies back congruence modularity;
actually, \eqref{ric}
implies $2k+2$-modularity (perhaps the bound $2k+2$ can be improved).
Indeed, taking $\Theta=  \alpha \gamma  \circ \beta \circ \alpha \gamma $
in \eqref{ric} we get
\begin{multline*}\labbel{} 
\alpha ( \beta \circ \alpha \gamma \circ \beta )
\subseteq
\alpha ( \alpha \gamma  \circ \beta \circ \alpha \gamma \circ \beta \circ \alpha \gamma )
 =
 \alpha ( \Theta \circ \Theta )
\subseteq ^{\eqref{ric}} 
\\
( \alpha \Theta) ^{k}=
(\alpha \gamma  \circ \alpha \beta \circ  \alpha \gamma )^{k}
= \alpha \gamma  \circ _{2k+1} \alpha \beta 
\subseteq 
 \alpha \beta  \circ _{2k+2} \alpha \gamma .
\end{multline*}      

Theorem  \ref{proprelb} shows that the best value of $k$
in \eqref{ric} for a congruence modular variety $\mathcal V$ 
is not determined by   
 the Day modularity level of $\mathcal V$
(and it is not determined by the J{\'o}nsson distributivity level, either,
for congruence distributive varieties).
Indeed, both $\mathcal B$ and $\mathcal N_4$
are $5$-modular, not $4$-modular,
$4$-distributive and not $3$-distributive,  
but the best value for $k$  in \eqref{ric}
is $4$ for $\mathcal B$ and $3$ for $\mathcal N_4$:
just
take $n=2$ in Theorem   \ref{proprelb}. 
In particular,
the variety  $\mathcal N_4$
shows that    \eqref{ric}
for some given $k$ 
does not imply 
$k+1$-modularity. 
\end{remark}

\begin{remark} \labbel{tolmodrel}
Since the identity
\begin{equation}\labbel{ar}     
\alpha( R \circ R) \subseteq (\alpha R)^k 
   \end{equation}
implies \eqref{ric}, then, by the above remark, 
the identity \eqref{ar} 
implies congruence modularity.
This  follows also directly from Lemma \ref{rm}.
As we mentioned in the introduction, it is shown  in \cite{ricm}
that the converse holds, that is, every congruence modular variety 
satisfies \eqref{ar}, 
 for some $k$.
Again, this argument from \cite{ricm}
relies heavily on \cite{adjt}.

As in Remark \ref{tolmod}, we get that
Theorem  \ref{proprelb}
shows that the best value of $k$
in \eqref{ar} for a variety $\mathcal V$ 
is not determined by   
 the Day modularity level of $\mathcal V$.
It is likely that there is some variety $\mathcal V$
such that the best value of $k$ in  
\eqref{ric} 
is strictly smaller than the best value of $k$
in \eqref{ar}.
The varieties considered here do not 
furnish a  (counter-)example for this possible inequality.
 \end{remark}

\section{Near-unanimity and edge terms} \labbel{nearu}

Of course, it is an interesting problem to evaluate 
the distributivity spectra of other congruence distributive varieties. 
Varieties with a near-unanimity term appear to be 
a significant test case.
Recall that
$R$, $S$, $T, \dots $ denote reflexive and admissible relations,
  $\Theta$ denotes a tolerance
and $ \overline{R \cup S} $ denotes the smallest  
 admissible relation containing both 
$R  $ and $  S$.

\begin{proposition} \labbel{sch}
Suppose that $m \geq 1$ 
and that the variety $\mathcal V$ has an $m+2$-ary near-unanimity term.
Then, for every $n \geq 2$, the variety  $\mathcal V$ satisfies
\begin{align}
\labbel {15} 
\alpha  (\beta \circ_n \gamma  )
&\subseteq 
\alpha  \beta   \circ _{mn} \alpha  \gamma      && \text{for $n$ even,}   
\\
\labbel {16}
\alpha  (\beta \circ_n \gamma  )
&\subseteq 
\alpha \beta   \circ _{1+m(n-1)} \alpha  \gamma      && \text{for $n$ odd,}   
\\
\labbel{17}
 \Theta   (R \circ_n R) &\subseteq  ( \Theta   R )^{1+m(n-1)}   
 && \text{for every $n$.}
\end{align}    

 Taking $n=2$  in equation \eqref{15}
we get that $\mathcal V$ is $2m$-distributive \cite{M}. 
Taking
$n=3$  in equation \eqref{16} with $ \alpha \gamma$ 
in place of $\gamma$,
we get that  
$\mathcal V$ 
is $2m+1$-modular. 
\end{proposition}

In fact, we shall prove a more general result.

\begin{proposition} \labbel{schbis}   
Under the assumptions in Proposition \ref{sch},  $\mathcal V$ satisfies
\begin{equation}\labbel{14}    
\begin{aligned}
 \Theta  (R_1 \circ  R_2 \circ \dotsc \circ R_n)  \subseteq 
\Theta  R_1 \circ \Theta  & R_2 \circ \dotsc \circ \Theta  R _{n-1} \circ
\Theta  (\overline{R_n  \cup R_1}) 
{\, \circ  \ }  
\\
  \Theta  & R_2 \circ \dotsc \circ \Theta  R _{n-1}  \circ
\Theta  (\overline{R_n  \cup R_1}) 
{\, \circ  \ }  
\\
\dotsc  \Theta  & R_2 \circ \dotsc \circ \Theta  R _{n-1} \circ
\Theta  R_n,
\end{aligned}
 \end{equation}   
with $m$ lines, that is, with a total of  $ 1 + m (n- 1)  $ factors
(a total of $  m (n- 1)  $ occurrences of $\circ$)
 on the right-hand side.
 \end{proposition} 

Equation \eqref{14}
should read
$\Theta  (R_1  \circ \dotsc \circ R_n)  \subseteq 
\Theta  R_1 \circ \Theta   R_2 \circ \dotsc \circ \Theta  R _{n-1} \circ
\Theta  R_n$
when $m=1$ and  
$\Theta  (R_1 \circ \dotsc \circ R_n)  \subseteq 
\Theta  R_1  \circ \dotsc \circ \Theta  R _{n-1} \circ
\Theta  (\overline{R_n  \cup R_1}) \circ \ 
 \Theta   R_2 \circ \dotsc  \circ
\Theta  R_n$
  when $m=2$.  

\begin{proof}[Proofs] 
Suppose that
$u$ is an $m+2$-ary near-unanimity term, 
$a \mathrel { \Theta  } d $
and
$a \mathrel { R_1 } c_1 \mathrel { R_ 2} c_2 \dots c_{n-2} 
 \mathrel { R_{n-1} } c_{n-1} \mathrel { R_ n} d $.
Then  
\begin{multline*} 
a= u(a, \dots, a, a, a, d) \mathrel { R_1}  u(a, \dots, a, a, c_1, d)
\mathrel { R_2}  u(a, \dots, a, a, c_2, d) \dots 
\\
\mathrel { R_{n-1}}  u(a, \dots, a, a, c_{n-1}, d)
 \mathrel { \overline{R_n  \cup R_1}}  u(a, \dots, a, c_1 , d, d)
\mathrel { R_2} u(a, \dots, a, c_2, d, d) 
\\
\dots
\\
\mathrel { R_{n-1}}  u(a, a, c_{n-1}, d, \dots, d)
 \mathrel { \overline{R_n  \cup R_1}}  u(a, c_1 , d, d, \dots, d)
\mathrel { R_2}  u(a, c_2 , d, d, \dots, d) 
\\
\dots
\mathrel { R_{n-1}}  u(a, c_{n-1} , d, d,  \dots, d) 
\mathrel { R_n } u(a, d , d, d, \dots, d) =d.
\end{multline*}    

In the above chain of relations we  have only used
a minimal part of 
the assumption that $u$ is a near-unanimity term: we have used 
 only the two special cases in which 
the ``dissenter'' is either the first or the last element. 
The full assumption will be used in order  to show 
that all the above elements are $\Theta$-related.
Were $\Theta = \alpha $ a congruence, this would be trivial, since
\begin{multline*}
u(a,\dots, a, c_j, d, \dots, d) \mathrel \alpha   u(a,\dots, a, c_j, a, \dots, a) =a =
\\ 
u(a,\dots, a, c_k, a, \dots, a) \mathrel \alpha  u(a,\dots, a, c_k, d, \dots, d) ,
 \end{multline*}   
 for all pairs of indices $j$ and $k$ and 
where the $c_j$'s and 
the $c_k$'s can occur in any pair of possibly distinct positions.

Notice that the case in which $\Theta$ 
is a congruence is enough in order to prove 
\eqref{15} and
 \eqref{16} in Proposition \ref{sch}.
Formally, \eqref{15} does not follow from
\eqref{14}; however in \eqref{14}
we can replace each occurrence of $ \Theta  (\overline{R_n  \cup R_1})$
by $ \Theta   R_n \circ \Theta  R_1$.
Indeed, say,   
$ u(a, \dots, a, a, c_{n-1}, d)
 \mathrel { \Theta   R_n}
 u(a, \dots, a, a, d, d)
 \mathrel { \Theta  R_1  }  u(a, \dots, a, c_1 , d, d)$. 

It remains to consider the case in which 
$\Theta$ is only supposed to be a tolerance.
The argument resembles a proof in 
Cz\'edli and Horv\'ath
\cite{CH}. 
As above, we shall show that any two  elements in the above chain 
(disregarding their ordering, that is, slightly more than required)
are 
$\Theta $-related.
Indeed, 
 \par
\vskip-10pt
{\par \small
\begin{align*}
&\quad u(a, \dots, a, c_j, d, \dots, d)=
\\
&u(u(\dots a, c_k, a \dots ), \dots
 u( \dots a, c_k, a \dots ),
 c_j, 
 u( \dots d, c_k, d \dots ),
 \dots u( \dots d, c_k, d \dots )) \mathrel { \Theta  }
\\[-5pt]
&u(u(\dots a, c_k, \overset{|}{d} \dots), \dots
u( \dots a, c_k, \overset{|}{d} \dots),
 c_j,
 u( \dots  \overset{\raisebox{2pt}{\scriptsize$|$}} {a}, c_k,
{ d} \dots ), \dots 
u( \dots \overset{\raisebox{2pt}{\scriptsize$|$}} a, c_k, {d} \dots )) =
\\
& \quad u(a, \dots, a, c_k, d, \dots, d),
 \end{align*}   
\par}
\noindent
again, for all pairs of indices $j$ and $k$ and
where the $c_j$'s and 
the $c_k$'s can occur in any pair of possibly distinct positions.
In the above-displayed formula
 the vertical bars link distinct  $\Theta$-related elements 
and, in order to keep the formula within a reasonable length,
we have written, say,
$u(\dots a, c_k, d \dots)$
in place of
$u(a,a,\dots,a, a, c_k, d, d, \dots, d,d)$.      
In conclusion, we get that 
$(a,d)$ belongs to the right-hand side of \eqref{14}. 
\end{proof}  

Equations \eqref{15} and
 \eqref{16}  show that a variety with a near-unanimity  term is congruence 
distributive, a result  originally due to Mitschke \cite{M}.
The above proof seems simpler than the one from \cite{M}
 and uses folklore ideas. 
Cf., e.~g., 
 Kaarli and Pixley \cite[Lemma 1.2.12]{KP},
whose proof is credited to E.\ Fried.
Notice that here and in \cite[Lemma 1.2.12]{KP}, as well,  
it is not necessary to use J{\'o}nsson's
characterization \cite{J} of congruence distributive varieties.  

 From a more recent point of view,
\eqref{14}  might be seen as a combination
of two observations.
First, the fact
that 
a near-unanimity term
easily  yields a set of
directed J{\'o}nsson terms;
see, e.~g., Barto and Kozik \cite[Section 5.3.1]{BK}.
Second, the observation in \cite{jds}  that directed J{\'o}nsson terms
not only imply congruence distributivity, but  also imply
certain similar relation identities.
The technical idea of merging $R_n$ and $R_1$,
so as to obtain a smaller number of factors in \eqref{14} and
\eqref{17} 
 seems new, at least  in the present context.    

We do not know whether we can replace 
the tolerance $\Theta$ by an admissible relation $T$ 
in \eqref{14} and  \eqref{17} 
(with or without the same number of factors on the right).
Apart from this, the variety 
of lattices shows that
the results in Proposition \ref{sch}
are best possible when $m=1$. 
Since $u$ 
is a $4$-ary near-unanimity term in  $\mathcal N_4$, then
Theorems \ref{bds},  \ref{pari} and
\ref{proprelb} show that the values of the indices
on the right-hand sides of 
  Proposition \ref{sch}  are best possible when 
$m=2$.

We now turn to edge terms,
 an 
important generalization of near-unanimity terms.
Berman, Idziak, Markovi\'c, 
  McKenzie,  Valeriote and Willard \cite{bimm}
have introduced edge terms in \cite{bimm}, 
providing equivalent characterizations for their existence.
Further characterizations have been found by
Kearnes and Szendrei in \cite{KS}.

If $k \geq 2$, a $k+1$-ary term $t$ is a 
\emph{$k$-edge term} 
for some variety $\mathcal V$ 
if the equations 
$x=t(y,y,x,x, \dots, x)=t(x,y,y,x, \dots, x)$ hold 
and, moreover, 
all the equations of the form 
$x=t(x,x,x, \dots, y, \dots  )$
hold in $\mathcal V$, where a single occurrence of $y$ appears 
in any place after the third place, surrounded by
$x$'s elsewhere.
 We have used here the formulation 
from \cite{KS} in which  the first two places
are exchanged. 
Notice that a $k$-ary  near-unanimity term 
becomes a $k$-edge term by adding a dumb
variable at the second place.
The following proposition 
provides, among other, still another proof that varieties with an edge term
are congruence modular.

\begin{proposition} \labbel{edge} 
If $k \geq 3$ and  $\mathcal V$ has a 
$k$-edge term, then $\mathcal V$ satisfies
\begin{align} \labbel{e1}   
\Theta (R \circ R) &\subseteq (\Theta R )^{k-1} 
\qquad\qquad \text{ and, more generally,}  
\\
\labbel{e2}
\Theta (R \circ S) &\subseteq  
\Theta R \circ  (\Theta (\overline{R \cup S} ) )^{k-2}. 
 \end{align} 
Thus $\mathcal V$ is $2k-3$-modular by Lemma \ref{rm}(3).
\end{proposition}

 \begin{proof}
The proof is similar to the proof of 
Proposition \ref{sch}, with  a variation
``near the edge''.
Equation \eqref{e1} is the particular case
$R= S$ of equation \eqref{e2}, hence it is enough to prove the
latter.    
If $a \mathrel { \Theta} d$
and $a \mathrel {R } c \mathrel { S } d $,
then
\begin{align*} \labbel{}
a= t&(a,a,a,a,a,a, \dots, a,a,a,d) \mathrel { R}  
\\[-6pt]
t&(a,a,a,a,a,a, \dots, a,a,
\overset{\raisebox{2pt}{\scriptsize$|$}}{c},
d)  \mathrel {\overline{R \cup S}}
\\[-7pt]
t&(a,a,a,a,a,a, \dots, a,
\overset{\raisebox{3pt}{\scriptsize$|$}}{c},
\overset{|}{d},d)  \mathrel {\overline{R \cup S}} \ \dots
\\
\dots \  t&(a,a,a,a,c,d, \dots, d,d,d,d)  \mathrel {\overline{R \cup S}} 
\\[-7pt]
t&(a,a,a,
\overset{\raisebox{3pt}{\scriptsize$|$}}{c},
\overset{|}{d},
d, \dots, d,d,d,d)  \mathrel {\overline{R \cup S}} 
\\[-7pt]
 t&(a,\hspace{1pt}\overset{\raisebox{3pt}{\scriptsize$|$}}{c},
\hspace{1pt}\overset{\raisebox{3pt}{\scriptsize$|$}}{c},
\overset{|}{d},
d,d, \dots, d,d,d,d)  \mathrel {\overline{R \cup S}}
\\[-7pt]
 t&(\hspace{1pt}\overset{\raisebox{3pt}{\scriptsize$|$}}{c},
\hspace{1pt} c,
\overset{|}{d},
d,d,d, \dots, d,d,d,d)  = d 
 \end{align*}   

The proof that all the above elements are $\Theta$-related is
similar to the corresponding proof in \ref{sch}.
For example,   

\begin{align*} \labbel{}
 t(a,c,c,d ...) =   t&(t(a,a,a,c,a...), c,c, t(d,d,d,c,d...) ...) \mathrel { \Theta }  
\\[-6pt]
 t&(t(a,a,a,c,{\overset{|}{d}}...), c,c, 
 t(\overset{\raisebox{3pt}{\scriptsize$|$}}{a},
\overset{\raisebox{3pt}{\scriptsize$|$}}{a},
\overset{\raisebox{3pt}{\scriptsize$|$}}{a},
c,d...)...) = 
t(a,a,a,c,{d}...). \qedhere
\end{align*}    
\end{proof}

\begin{remark} \labbel{gg}   
Merging the proofs of Propositions
\ref{schbis} and  \ref{edge}
we get that 
if  $k \geq 4$ 
and   $\mathcal V$ has a $k$-edge term,
then, for every $n \geq 2$, the variety  $\mathcal V$ satisfies
\begin{equation*}\labbel{14e}    
\begin{aligned}
 \Theta (R \circ S) (R_1 \circ \dotsc \circ R_n)  \subseteq 
\Theta  R_1 \circ \Theta  & R_2 \circ \dotsc \circ \Theta  R _{n-1} \circ
\Theta  (\overline{R_n  \cup R_1}) 
{\, \circ  \ }  
\\
  \Theta  & R_2 \circ \dotsc \circ \Theta  R _{n-1}  \circ
\Theta  (\overline{R_n  \cup R_1}) 
{\, \circ  \ }  
\\
\dotsc  \Theta  & R_2 \circ \dotsc \circ \Theta  R _{n-1} \circ
\Theta  (\overline{R_n  \cup R}) \circ 
\Theta  (\overline{R  \cup S}),
\end{aligned}
 \end{equation*}   
with $k-3$ lines.
\end{remark}

\section{Further remarks} \labbel{fur} 

Recall that $\alpha$, $\beta$, $\gamma,\dots$
denote congruences, $\Theta$ denotes a tolerance  and 
$R$, $S$, $T, \dots $ denote reflexive and admissible relations.

\begin{remark} \labbel{fact}   
It is standard and easy to
show that varieties with a majority term satisfy
$ T( R \circ S) \subseteq TR \circ TS$. 
See, e.\ g., 
 \cite{jds,uar}. Since the composition of two admissible
relations is still admissible, then, by 
substitution and 
an easy induction, we get that if  $\mathcal V$ has 
a majority term, then, for every $n \geq 2$, the variety  $\mathcal V$ satisfies
$ T  (R_1 \circ  R_2 \circ \dotsc \circ R_n)  \subseteq 
T  R_1 \circ T   R_2 \circ \dotsc \circ T  R _{n}$.

Moreover, by taking 
$T = (R_3 \circ R_2)(R_1 \circ R_3)$, we get that a variety 
with a majority term satisfies
$(R_1 \circ R_2)(R_3 \circ R_2)(R_1 \circ R_3)
\subseteq R_1R_3 \circ R_1R_2 \circ R_1R_2\circ R_2R_3$.
\end{remark}  

\begin{remark} \labbel{majari}
In passing, we notice the curious fact that
while, of course, 
the identity 
$\alpha( \beta \circ \gamma )= \alpha \beta \circ \alpha \gamma $ 
for congruences
is equivalent to the existence of a majority term,
on the other hand, the identity
\begin{equation}\labbel{arit}     
(\alpha \circ \delta )( \beta \circ \gamma )=
 \alpha \beta \circ \alpha \gamma \circ \delta \beta \circ \delta \gamma 
  \end{equation}
is equivalent to arithmeticity.
Notice that we are assuming equality, not just inclusion.

To prove the claim, assume
equation \eqref{arit} and expand the product in two ways, getting
$\alpha \beta \circ \alpha \gamma \circ \delta \beta \circ \delta \gamma =
\alpha \beta \circ  \delta \beta \circ \alpha \gamma  \circ \delta \gamma $.
Then taking $\gamma = \alpha $ and $\delta= \beta $    
we get
$ \alpha \circ \beta = \alpha \beta \circ \alpha \circ  \beta \circ \alpha \beta  =
\alpha \beta \circ \beta \circ \alpha   \circ \alpha \beta = \beta \circ \alpha $, 
that is, congruence permutability.
On the other hand, by taking 
$\beta = \gamma $ in \eqref{arit},
we get $2$-distributivity, that is, a majority term.   
It is well-known that arithmeticity 
is equivalent to congruence permutability
together with the existence of a majority term, hence our claim follows.
 \end{remark}

\begin{remark} \labbel{fv3}
(a)
It follows from \cite{J} 
and is by now standard that, for every $m$,  some variety $\mathcal {V}$ 
satisfies the congruence identity 
\begin{equation}\labbel{md}
    \alpha( \beta \circ \gamma ) \subseteq \alpha \beta \circ_m \alpha \gamma
 \end{equation}
if and only if \eqref{md}  holds in $\mathbf F _{ \mathcal V } ( 3 ) $,
the free algebra in $\mathcal {V}$ generated by three elements.
Actually, what is relevant in the above sentence 
is the left-hand side of 
\eqref{md}; the  sentence is true whenever we replace the
right-hand side of \eqref{md} with any expression
in function of $\alpha$, $\beta$, $\gamma$ 
constructed by using intersection, composition
and even transitive closure.

Henceforth, a possible way to check  
whether some variety $\mathcal {V}$ satisfies 
\eqref{md}, or even  many  related identities,
 is to check \eqref{md} in $\mathbf F _{ \mathcal V } ( 3 ) $.
Actually, if $x$, $y$ and  $z$ are the generators
 of $\mathbf F _{ \mathcal V } ( 3 ) $,
it is enough to check \eqref{md}
in the special case when $\alpha$, $\beta$ and $\gamma$   
are the congruences generated, respectively,
 by the pairs $(x,z)$,  $(x,y)$, and $(y,z)$.
This is  classical, by now. 
As we shall mention in (e) below,
  the above procedure is not the simplest way to check
\eqref{md}, or to check congruence distributivity; however, it is
the one relevant to the following discussion.

(b)
Let us compute, for example,
$\mathbf F _{ \mathcal B^d} ( 3 ) $.
Since $\mathbf F _{ \mathcal B^d} ( 3 ) $
is naturally embedded into $\mathbf F _{ \mathcal D } ( 3 ) $,
where $\mathcal {D}$ is the variety of distributive lattices,
 it is easy to see that the elements of $\mathbf F _{ \mathcal B^d} ( 3 ) $
are
\begin{equation*}\labbel{fv3b}
x, \qquad x(y+z), \qquad xy, \qquad xyz,
  \end{equation*}       
together with the elements arising from 
all the possible permutations of $x$, $y$ and $z$. 
Cf.\ \cite{B}.
The elements in the above list are exactly also the elements
of $\mathbf F _{ \mathcal N_4 ^d } ( 3 ) $, since 
 $ F _{ \mathcal B^d} ( 3 ) $ is closed under 
$u$. Indeed, if 
$a \in  F _{ \mathcal B^d} ( 3 ) $,
then, by the above description,  either $a \leq x$,
or  $a \leq y$, or $a \leq z$.
Hence if
$a_1, a_2, a_3, a_4 \in  F _{ \mathcal B^d} ( 3 ) $,
then at least two elements are less than or equal to some generator, say,
 $a_1, a_2 \leq x$, thus $u(a_1, a_2, a_3, a_4) \leq x$, hence 
 $u(a_1, a_2, a_3, a_4 )\in  F _{ \mathcal B^d} ( 3 ) $.
Of course, the above argument would fail,
were we considering  $\mathbf F _{ \mathcal V } ( 4 ) $
in place of $\mathbf F _{ \mathcal V } ( 3 ) $.
 
In a sense, computing $\mathbf F _{ \mathcal V } ( 3 ) $  
is a way to see that $\mathcal N_4 ^d$ and $\mathcal B^d$
satisfy exactly the same identities of the form 
$ \alpha ( \beta \circ \gamma ) \subseteq something$ for congruences,
for many possible variations on  the expression on the right.  
However, there is a subtle related issue 
we are going to discuss soon.

(c) Consider now relation identities
of the form 
\begin{equation}\labbel{mdr} 
T(R \circ S) \subseteq TR \circ_m TS.
  \end{equation}    
Again, many variations are possible, 
including letting some variable be a congruence or a tolerance.
It is still true that some variety $\mathcal {V}$ 
satisfies \eqref{mdr} if and only if  
\eqref{mdr}  holds in $\mathbf F _{ \mathcal V } ( 3 ) $.

The argument is standard but not very usual.
To prove the non trivial implication, let 
$\mathbf F _{ \mathcal V } ( 3 ) $
be generated by $x$, $y$ and  $z$
and  $R$, $S$ and $T$   
be the smallest reflexive and admissible relations containing, respectively,
 the pairs  $(x,z)$,  $(x,y)$ and $(y,z)$.
Since 
$(x,z) \in T(R \circ S)$, then,
if  \eqref{mdr}  holds, there are elements
$t_0, \dots, t_m \in  F _{ \mathcal V } ( 3 ) $   
witnessing $(x,z) \in TR \circ_m TS$.
Since we are working in $\mathbf F _{ \mathcal V } ( 3 ) $, 
the $t_i$'s correspond to ternary terms of $\mathcal {V}$.
What does it mean, say, that
$t_i \mathrel R t_{i+1}  $?
It is easy to see that 
$R =\{ (w(x,y,z, x), w(x,y,z, z))  \mid
 w \text{ a $4$-ary term of  }  \mathcal {V} \}$. 
Hence 
$t_i (x,y,z) =w_i(x,y,z, x) $
and 
$ w_i(x,y,z, z) = t_{i+1} (x,y,z)$, for some
 $4$-ary term  $w_i$.
Similarly, the $S$- and $T$-relations 
are witnessed by certain other 
$4$-ary terms.
Once we have found  appropriate terms
and appropriate equations, it is standard to see that 
they witness that 
\eqref{mdr} holds in $\mathcal {V}$.
See \cite[Proposition 3.7]{uar} for full details. 

(d) We now see an essential difference between
the observations in   (a) and (c) above.
By (b), we have 
 $F _{ \mathcal N_4 ^d } ( 3 ) =  F _{ \mathcal B^d} ( 3 ) $;
nevertheless, we have seen in Theorem \ref{pari}  that
$\mathcal N_4 ^d  $ and 
 $ \mathcal B^d$
 \emph{do not}
satisfy the same identities of the form 
\eqref{mdr}, in the sense that the best 
possible indices on the right are not the same.
At first, this might generate some
perplexity, but in the end the explanation is easy.
The point is that,
when considering congruence identities of the form
\eqref{md},
only  ternary terms 
  are relevant;
in other words, only the elements of $\mathbf F _{ \mathcal V } ( 3 ) $  
are relevant.
On the other hand, as shown in remark (c)
above, though the validity of 
\eqref{mdr} is checked in $\mathbf F _{ \mathcal V } ( 3 ) $,
the relevant terms, in this case, are the $4$-ary ones.
Thus we have to deal with the algebraic structure of 
$\mathbf F _{ \mathcal V } ( 3 ) $, not simply with the set of its elements.
Notice that, were we considering tolerance identities, rather than
relation identities, we should deal with $5$-ary terms.  

(e) In spite of the considerations in 
(a) above, working in $\mathbf F _{ \mathcal V } ( 3 ) $ 
is not the simplest way in order to check
the validity of some \emph{congruence}
identity of the form \eqref{md}.
Since J{\'o}nsson's equations 
\cite{J} are essentially two-variable equations, 
a variety $\mathcal {V}$ is $m$-distributive 
if and only if  $\mathbf F _{ \mathcal V } ( 2 ) $  
generates an $m$-distributive variety
(warning: it might happen that $\mathbf F _{ \mathcal V } ( 2 ) $  
is congruence distributive, but the variety it \emph{generates} is not!)
In fact, it is enough to check 
$m$-distributivity in an appropriate subalgebra of 
  $(\mathbf F _{ \mathcal V } ( 2 ) )^3$ and,
for finite idempotent algebras, there are even 
computationally more effective methods to check congruence
distributivity.
Cf.\ Freese and Valeriote \cite{FV}.
 \end{remark}

Now consider, in general, a variety $\mathcal {R}$ 
which is obtained by taking polynomial reducts of lattices.
Then $\mathbf F _{ \mathcal R } ( 2 ) $
is a polynomial reduct of $\mathbf F _{ \mathcal L} ( 2 ) $,
where $\mathcal {L}$ is the variety of lattices.     
But $\mathbf F _{ \mathcal L} ( 2 )  = \mathbf F _{ \mathcal D } ( 2 ) $,
where $\mathcal {D}$ is the variety of distributive lattices,
hence the variety generated by 
$\mathbf F _{ \mathcal R } ( 2 ) $
is a polynomial reduct of \emph{distributive}
lattices. This is one of the main arguments
in the proof of \cite[Theorem 2]{B}
(warning: it is not necessarily the case that 
$\mathcal {R}$ itself is a reduct of the variety of distributive lattices:
this applies only to the subvariety of $\mathcal {R}$ generated by $\mathbf F _{ \mathcal R } ( 2 ) $).
Exactly as in \cite{B}, some results provable for $\mathcal B$ extend to \emph{every}
variety which is 
a congruence distributive  polynomial reduct of some variety of lattices.
This is the content of the next
theorem, where we shall also show that 
there are limitations to the counterexamples which
can be furnished by polynomial reducts
of Boolean algebras.

If $P$ is a set of lattice terms and $\mathbf L$ is a lattice, we denote 
by $\mathbf L_P$ 
the algebra with base set $L$ 
and with, as basic operations, those induced by the terms of $P$.
If $\mathcal V$ is a variety of lattices, we let $\mathcal V_P$
be the variety generated by all algebras 
$\mathbf L_P$, with $\mathbf L$ varying in $\mathcal V$.
If $\mathcal W =\mathcal V_P$, for some $P$,
we shall say that $\mathcal W$  is a \emph{polynomial reduct}
of $\mathcal V$. 

\begin{theorem} \labbel{prl}
(1) If the variety $\mathcal W$ is a congruence distributive polynomial reduct of  
some variety of lattices, then $\mathcal W$ satisfies
\begin{align*}   \labbel {9gen}
 \Theta  (R \circ_n S) &\subseteq  \Theta  R  \circ _{2n} \Theta  S    
 && \text{for $n$ even, and}   
\\
 \Theta  (R \circ_n S) &\subseteq  
 (\Theta  R  \circ \Theta  S \circ {\stackrel{n}{\dots}} \circ  \Theta R ) {\circ}
 (\Theta  R  \circ \Theta  S \circ {\stackrel{n}{\dots}} \circ  \Theta R ) 
 && \text{for $n$ odd.}   
 \end{align*}

(2) If the variety $\mathcal W$ is a congruence distributive polynomial reduct of  
the variety of \emph{distributive} lattices, then $\mathcal W$ satisfies
 the equations \eqref{9} and \eqref{10} in  Theorem  \ref{pari}.
In other words, we may allow $\Theta$ to be a reflexive and admissible relation 
in the identities in (1) above.

(3)
If $\mathcal W$ is a congruence 
modular polynomial reduct of the variety of Boolean 
algebras, then 
either $\mathcal W$ has a majority term, or
$\mathcal W$ has a Maltsev  term, or
$\mathcal W$ interprets $\mathcal B$.
If in addition $\mathcal {W}$ is
congruence distributive, 
then $\mathcal W$ satisfies
 the equations \eqref{9} and \eqref{10} in  Theorem  \ref{pari}. 
\end{theorem}

 \begin{proof} 
(1) First, notice that in the proof that Baker's variety 
$\mathcal B$ satisfies  
the identities \eqref{1} and \eqref{4}
  in Theorem \ref{bds} 
we have only used the equations
\begin{equation}\labbel{b}     
x=b(x,x,y)=b(x,y,x) \quad \text{  and } \quad
b(x,y,y)=b(y,x,x)
  \end{equation}
and that, as we mentioned, the proof works
also when $\beta$ and $\gamma$ are admissible relations.
Using the idea from
Cz\'edli and Horv\'ath
\cite{CH}, an idea  we have already used above, we can replace $\alpha$ in 
\eqref{1} and \eqref{4} by a tolerance $\Theta$.
Indeed, if $\Theta$ 
is a tolerance and $a \mathrel \Theta d$, then from the
equation $x=b(x,y,x) $ we get
\begin{align*}
 b(a,c_{j}, d) =  & b( b(a,c_{k},  a),c_{j},  b(d,c_{k}, d)) \mathrel \Theta
\\[-6pt]
& b( b(a,c_{k}, \overset{|}{d}),
c_{j},  b(\stackrel{\raisebox{2pt}{\scriptsize$|$}}{a},c_{k}, d))
 = b(a,c_{k}, d),
 \end{align*}   
for all pairs of indices $j$ and $k$.
Notice that we have showed a little more than requested, namely, that all the elements
from the list in equation \eqref{B} in the proof of 
\ref{bds} are $\Theta$-related. 

Hence if a variety 
has a term satisfying \eqref{b},
then the conclusion in (1) holds.
We  shall show that a variety satisfying the assumptions
in (1) either has a majority term, or a term
satisfying \eqref{b}. 
The argument goes exactly as in 
the proof of \cite[Theorem 2]{B}, as we shall see. 
Since the equations \eqref{b}
depend only on two variables, then, for some given term $t$, 
they hold
in $\mathcal W$ if and only if they hold in the
free algebra $\mathbf F _{ \mathcal W } ( 2 ) $  in $\mathcal W$ generated by two elements. 
Suppose that $\mathcal W =\mathcal V_P$. Since the 
free lattice generated by two elements is 
$\mathbf C_2 \times \mathbf C_2$, where 
$\mathbf C_2= \{ 0,1 \} $ is the chain with two elements,
then   $\mathbf F _{ \mathcal W } ( 2 ) $  
is a (possibly improper) subalgebra of
$(\mathbf C_2)_P \times (\mathbf C_2)_P$.
Hence the equations \eqref{b} hold in $\mathcal W$ 
if and only if they hold in $(\mathbf C_2)_P$, if and only if 
they hold in  the variety $\mathcal W'$ generated by $(\mathbf C_2)_P$.
If $\mathcal W$ is congruence distributive, then 
$\mathcal W'$  is congruence distributive, too.
Now the  free algebra 
$\mathbf F _{ \mathcal W' } ( 3 ) $  
in $\mathcal W'$
generated by three elements $x$, $y$ and  $z$
can be seen as a subalgebra of 
the free distributive lattice generated by
$x$, $y$ and  $z$.
Baker \cite[proof of Theorem 2]{B}
shows that $\mathbf F _{ \mathcal W' } ( 3 ) $  
must contain either the median    
$xy+xz+yz$, or 
the Baker element
$x(y+z)$ or its dual.
These element are given by a 3-ary term
$t$ of $\mathcal W'$ and the above arguments show that
in the former case $t$ is a majority term for $\mathcal W$,
while in the latter case 
$t$ satisfies the equations 
\eqref{b}.   
Notice that it is not necessarily the case that 
$t$ is interpreted as  
$xy+xz+yz$ or 
$x(y+z)$ throughout $\mathcal W$,
we only get that $t$ is either a majority term or 
satisfies \eqref{b}.
However this is enough, by  
Fact \ref{fact} in the former case, and by the comment
in the first paragraph of the proof in the latter case.
Hence (1) is proved. 

(2) is a particular case of the last statement of (3),
however a direct proof along the lines of (1) is easy.
Under the additional assumption, we
can  argue directly in $\mathcal W$, 
rather than in $\mathcal W'$,
hence in the present case $t$ can be  actually interpreted as  
$xy+xz+yz$ or 
$x(y+z)$ or the dual throughout $\mathcal W$.
In the former case
Fact \ref{fact} is enough
and in the latter case the arguments in the proof of \eqref{nmg}
in Proposition \ref{moregen} 
  carry over.

(3)
Let us prove the first statement. If $\mathcal W$ is congruence modular,
then $\mathcal W$ has ternary \emph{directed Gumm terms},
as introduced in
Kazda,  Kozik, McKenzie, Moore
 \cite[p. 205]{adjt}.
See \cite[Theorem 1.1 (3)]{adjt}. 
We shall recall the equations that
directed Gumm terms  satisfy  
as soon as  needed. 
Obviously, a ternary term of $\mathcal W$ 
corresponds to a ternary Boolean term $t$,
 hence it is no loss of generality to assume that
$t(x,y,z)=a_1xyz+a_2xyz'+ a_3xy'z +\dots$,
where $'$ denotes complement and each 
$a_1$, $a_2$, \dots \ is either $0$ or $1$. 
The first term $d_1$ in the set of
directed Gumm terms  satisfies the equations $d_1(x,x,y)=x=d_1(x,y,x)$. 
Represent $d_1$ by a Boolean expression as above. 
By the first equation, the coefficients of $xyz $ and $   xyz'$
must be $1$ and the coefficients of  
$x'y'z $ and $  x'y'z'$ must be $0$.
By the second equation, the coefficients of     
$xyz $ and $  xy'z$ must be $1$ and  
the coefficients of 
$x'yz' $ and $ x'y'z'$ must be $0$.
Considering all the possibilities, 
one easily sees that $d_1$ is either the
majority term $xy+xz+yz$, or the Baker term, or the dual of 
the Baker term, or the first projection.
In all but the last case
we are done.
If $d_1$ is the first projection, 
then the equations for directed Gumm terms
give $d_2(x,x,y)=d_1(x,y,y)=x$
and   $x=d_2(x,y,x)$,
hence we can repeat the above argument for $d_2$.
Going on, if we either get a majority term, or a Baker term, or its dual,
we are done as above. Otherwise, all the $d_j$'s are first projections.
Then the remaining term $q$ in the set of directed Gumm terms
satisfies $q(x,x,y)=y$ and $q(x,y,y)=d_n(x,y,y)=x$,
hence $q$ is a Maltsev term for permutability.   

To prove the last statement, 
if $\mathcal W$ is congruence distributive, then we have \emph{directed
J{\'o}nsson terms}, to the effect 
that $q$ as above is the third projection
(or, simply, discard $q$ and ask for 
$d_n(x,y,y)=y$). 
Arguing as above, 
we get that  some $d_j$ satisfies all the equations
satisfied in distributive lattices  by either  the majority term or
by the Baker term,
hence,  again, either
Fact \ref{fact} or  the proof of \eqref{nmg}
apply. 
\end{proof}

We expect that \ref{prl}(2) might fail
if $\mathcal W= \mathcal V_P$,
when $\mathcal V$   is not the variety of distributive lattices.
In other words, we expect that 
(at least, without affecting the subscripts)
\ref{prl}(1) cannot be improved in such a 
way that admissible relations are taken into account everywhere.
However, we notice that in the proof of Theorem \ref{bds}
all the lattices we have considered are indeed distributive.
Hence, in view of \ref{prl}(2), in order to provide a counterexample that 
\ref{prl}(1) cannot be improved in the above sense, one should start with
a different and more complicated example, i.\ e., 
it is not enough to consider different lattice term operations
on the same set $B$ considered in the proof of 
\ref{bds}.

\smallskip 

{\scriptsize The author considers that it is highly  inappropriate, 
 and strongly discourages, the use of indicators extracted from the list below
  (even in aggregate forms in combination with similar lists)
  in decisions about individuals (job opportunities, career progressions etc.), 
 attributions of funds  and selections or evaluations of research projects. \par }

\end{document}